\documentclass{amsart}

\usepackage{amssymb}
\usepackage{amsfonts}
\usepackage{amssymb}
\usepackage{cite}
\usepackage{amsmath}
\usepackage{amsthm}
\usepackage{enumerate}
\usepackage{tabularx}
\usepackage{centernot}
\usepackage{mathtools}
\usepackage{stmaryrd}
\usepackage{amsthm,amssymb}
\usepackage{etoolbox}
\usepackage{tikz}
\usepackage{amssymb}
\usepackage{tablefootnote}
\usetikzlibrary{matrix}
\usepackage{graphics,graphicx}
\usepackage[all]{xy}
\usepackage{tikz-cd}
\usepackage{url}
\usepackage{marginnote}
\definecolor{mygray}{gray}{0.85}
\usepackage[linecolor=black,backgroundcolor=mygray,colorinlistoftodos,prependcaption,textsize=small]{todonotes}
\usepackage{xargs}

\renewcommand{\leq}{\leqslant}
\renewcommand{\geq}{\geqslant}

\newcommand{\CC}{\mathbb{C}}

\newcommand{\Cscr}{\mathcal C}
\newcommand{\Dscr}{\mathcal D}

\newcommand{\Gscr}{\mathcal G}

\newcommand{\Sscr}{\mathcal S}
\newcommand{\Uscr}{\mathcal U}

\def\icl{\rm icl}
\def\dcl{\rm dcl}
\def\cl{\rm cl}

\def\acl{\rm acl}
\def\diag{\rm diag}
\DeclareMathOperator{\ddcl}{cl^d}
\DeclareMathOperator{\ddclM}{cl^d_M}

\makeatletter
\def\subsection{\@startsection{subsection}{3}%
  \z@{.5\linespacing\@plus.7\linespacing}{.3\linespacing}%
  {\bfseries\centering}}
\makeatother

\makeatletter
\def\subsubsection{\@startsection{subsubsection}{3}%
  \z@{.5\linespacing\@plus.7\linespacing}{.3\linespacing}%
  {\centering}}
\makeatother

\makeatletter
\def\myfnt{\ifx\protect\@typeset@protect\expandafter\footnote\else\expandafter\@gobble\fi}
\makeatother

\newtheorem{theorem}{Theorem}[section]

\newtheorem{corollary}[theorem]{Corollary}
\newtheorem{definition}[theorem]{Definition}
\newtheorem{lemma}[theorem]{Lemma}
\newtheorem{claim}[theorem]{Claim}
\newtheorem{proposition}[theorem]{Proposition}
\newtheorem{example}[theorem]{Example}

\newtheorem{question}[theorem]{Question}

\newtheorem{fact}[theorem]{Fact}
\newtheorem{conclusion}[theorem]{Conclusion}
\newtheorem{remark}[theorem]{Remark}
\newtheorem{notation}[theorem]{Notation}

\newtheorem{definition/fact}[theorem]{Definition/Fact}

\def\bK{\mbox{\boldmath $K$}}
\def\bL{\mbox{\boldmath $L$}}
\def\Fscr{\mathcal F}
\def\Lscr{\mathcal L}

\newcommand{\pureindep}[1][]{%
  \mathrel{
    \mathop{
      \vcenter{
        \hbox{\oalign{\noalign{\kern-.3ex}\hfil$\vert$\hfil\cr
              \noalign{\kern-.7ex}
              $\smile$\cr\noalign{\kern-.3ex}}}
      }
    }\displaylimits_{#1}
  }
}

\newcommand{\indep}[2]{%
  \mathrel{
    \mathop{
      \vcenter{
        \hbox{%
\oalign{
\noalign{\kern-.3ex}\hfil$\vert$\hfil\cr
              \noalign{\kern-.7ex}
              $\smile$\cr\noalign{\kern-.3ex}
}
}
      }
}^{\!\!\!\!\!#2}_{\!\!\hspace{-0.1em}#1}
  }
}

\begin{document}

\begin{abstract}
A linear space is a system of points and lines such that any two distinct points determine a {\em unique} line; a Steiner $k$-system (for $k \geq 2$) is a linear space such that each line has size exactly $k$. Clearly, as a two-sorted structure, no linear space can be strongly minimal.  We formulate linear spaces in a (bi-interpretable) vocabulary $\tau$ with a single ternary relation $R$. We prove that for every integer $k$ there exist $2^{\aleph_0}$-many integer valued functions $\mu$ such that each $\mu$ determines a distinct strongly minimal Steiner $k$-system $\mathcal{G}_\mu$, whose algebraic closure geometry has all the properties of the {\em ab initio} Hrushovski construction. Thus each is a \mbox{counterexample to the Zilber Trichotomy Conjecture.}
\end{abstract}

\title{Strongly Minimal Steiner Systems I: Existence}
\thanks{The first author was partially supported (during a 2018 visit to Jerusalem) by European Research Council grant 338821 and by Simons
travel grant G5402. The second author was partially supported by European Research Council grant 338821.}

\author{John Baldwin}
\address{Department of Mathematics, Statistics, and Computer Science, University of Illinois at Chicago, Chicago, USA.}
\email{jbaldwin@uic.edu}

\author{Gianluca Paolini}
\address{Department of Mathematics ``Giuseppe Peano'', University of Torino, Via Carlo Alberto 10, 10123, Italy.}
\email{gianluca.paolini@unito.it}

\date{\today}
\maketitle



%

\section{Introduction}



Zilber conjectured  that every strongly minimal set was  (essentially) bi-\hspace{0.0000001cm}interpretable  either with a strongly minimal set whose associated $\acl$-geometry  was trivial or locally modular, or with an algebraically closed field.
Hrushovski
\cite{Hrustrongmin} refuted that conjecture by a seminal extension of the Fra{\" \i}ss\'{e}
construction of $\aleph_0$-categorical theories as `limits' of finite structures to construct strongly minimal (and so $\aleph_1$-categorical) theories. In this paper we modify Hrushovski's method to construct $2^{\aleph_0}$-many strongly minimal
Steiner systems that also violate Zilber's conjecture. The examples arising from Hrushovski's construction have been seen as pathological, and there has been little work exploring the actual theories. The new examples that we construct here are infinite analogs of concepts that have
been central to  combinatorics for 150 years. But most of these investigations (e.g.\! \cite{BatBe, ColRosa, ReidRosa}) focus on finite systems.

Our construction of strongly minimal linear spaces via a Hrushovski construction might lead in
 two directions: (i) explore infinite Steiner systems investigating
 combinatorial notions appearing in such papers  as \cite{Cameronlinsp, CameronWebb,GanWer,Steinpnas}; (ii) search for further mathematically interesting
  strongly minimal sets with exotic geometries.
 This paper is an essential prerequisite for the sequel
  \cite{BaldwinsmssII}, where we address both issues by showing for $k>3$ that the examples here are {\em essentially unary\footnote{There is no parameter-free definable function depending on more than one variable.}}, expand the techniques used here to construct strongly minimal quasigroups, and extend the combinatorial analysis of \cite{CameronWebb}
  to those quasigroups.
%

Our construction combines methods from the theory of linear spaces/\hspace{0.0000001cm}combinatorics and model theory. 
 A linear space (Definition~\ref{linspace}) is a system of points and lines such that any two points determine a {\em unique} line. A Steiner $k$-system is a linear space  such that all lines have size $k$.
 We explain strong minimality below, and explore  its connection with Steiner systems in Section~\ref{sm}. 


The key ingredient of our construction is the development in \cite{Paoliniwstb} of a new model theoretic rank
function inspired by Mason's $\alpha$-function \cite{mason}, which arose in matroid theory. Using this new rank to produce a strongly minimal set requires a variant on the Hrushovski
construction \cite{Hrustrongmin} with several new features.

This is the first of a series of papers exploring these examples.
Here are the main results of this paper; they depend on definitions explained below.

\begin{itemize}
\item {\bf Theorem~\ref{bint}}: The one-sorted (Definition~\ref{taulin}) and two-sorted (Definition~\ref{linspace}) notions of linear space are bi-\hspace{0.0000001cm}interpretable.
\item {\bf Theorem~\ref{designcon}(2)}: For each  $k$, with $3 \leq k < \omega$, there are  $2^{\aleph_0}$-many strongly minimal theories $T_\mu$ (depending\footnote{The theory of course depends
   on the line length $k$; $k$ is coded by $\mu$ so we suppress the~$k$.} on an integer valued function $\mu$) of infinite linear spaces in the one-sorted vocabulary $\tau$ that are Steiner $k$-systems.
    \item {\bf Conclusion~\ref{getflat}}:
     Each theory $T_\mu$ admits weak\footnote{In view of Lemma~\ref{nlf} and Fact~\ref{thehammer}  our argument may, in very special cases, require naming finitely many constants to guarantee that $\acl(\emptyset)$ is infinite.} elimination of imaginaries, its geometry   is not locally modular, but it is CM-trivial and so it does not interpret a field. Thus, it violates Zilber's conjecture.
 \end{itemize}

 The last two results make sense only in the one-sorted vocabulary~$\tau$ (see below for a more detailed explanation of this).  This phenomena is symptomatic of the interplay among model theory, finite geometries and matroid theory.  Notions in these areas are `almost' the same. Sometimes `almost' is good enough and sometimes not.  The same intuitive structures are formalized in different vocabularies and in different logics depending on the field. Thus, the first task of this paper is to explain this interaction.
The first main result addresses this issue; further refinements on bi-interpretability appear in Section~\ref{12sort} and even more in \cite{BaldwinsmssII}.

We investigate here a new case where the structures have classical roots. Much of the current research on strongly minimal theories (as opposed for example to the strongly minimal sets discovered in differentially closed fields) focuses on classifying the attached $\acl$-geometry. Work of Evans, Ferreira,  Hasson, and Mermelstein \cite{EvansFerI,EvansFerII,MermelHas,Mermelthesis} suggests that up to arity
or more precisely, purity, (and modulo some apparently natural conditions\footnote{In \cite{EvansFerI,EvansFerI},  the class of finite structures is restricted only  by the dimension function and properties of $\mu$, that satisfy several technical conditions, which don't hold in some constructions in \cite{BaldwinsmssII}, as opposed to such axioms as `two points determine a line' here or the existence of a quasigroup structure in \cite{BaldwinsmssII}.})
any two $\acl$-geometries associated with strongly minimal Hrushovski constructions are
locally isomorphic.    This analysis is orthogonal to our program, which focuses on the particular strongly minimal theories constructed.

The naive observation that a plane has Morley rank 
$2$ motivated the construction in \cite{Baldwinasmpp} of an $\aleph_1$-categorical non-Desarguesian {\em projective plane} of Morley rank exactly $2$. The novelty of that result is the failure of the Desarguesian axiom; while the projective plane over $\CC$  has Morley rank $2$,  it is  `field-like' and so Desarguesian.
The result here complements  that example, weakening `projective plane' to `plane' (a linear space which admits the structure of a simple rank 3 matroid) while strengthening Morley rank $2$ to strongly minimal
(i.e.\! Morley rank $1$ and Morley degree $1$). And the examples
turn out to be Steiner systems.

%


A key difference from the finite situation is that   $k$-Steiner systems of finite cardinality $v$ occur only under strict number theoretic conditions on $v$ and $k$.  In contrast, for every $k$, we construct theories with countably many models in $\aleph_0$ and one in each uncountable power that are all Steiner $k$-systems.
But the number theory reappears when we attempt to find algebraic structures associated with these geometries. One goal is to coordinatize the Steiner systems by nicely behaved algebras. A substantial literature \cite{MR0094404,Steinpnas,GanWer,GanWer2} builds a correspondence between $k$-Steiner systems and certain varieties of universal algebras. But while this correspondence is  a
bi-interpretation for $k =3$, it does not rise to that level in general.  Indeed, for $k>3$, we show  \cite{BaldwinsmssII} that none of the strongly minimal Steiner systems constructed here interpret a quasi-group\footnote{A quasigroup is a structure $(A,*)$ such  that specification of any  two of $x, y, z$ in the equation $x * y = z$ determines
the third uniquely. This roughly corresponds to the current usage of groupoid. But, in the literature mentioned in the paragraph a groupoid is an algebra with a single binary function. }. We  also prove there  that for $q$ a prime power, and $V$ an appropriate variety, for each of our theories $T_\mu$ there is a theory $T_{\mu,V}$ of a strongly minimal quasigroup  in $V$ that interprets a $q$-strongly minimal Steiner system.
%
%

As already mentioned, most of the literature on linear spaces focuses on finite structures, but Cameron \cite{Cameronlinsp} asserts:

 \begin{quote}
 There is no theory of infinite linear spaces comparable to the enormous amount
known about finite linear spaces. This is due to two contrasting factors. First,
techniques which are crucial in the finite case (notably counting) are not available.
Second, infinite linear spaces are too easy to construct; instead of having to force our
configurations to ‘close up’, we just continue adding points and lines infinitely often!
The result is a proliferation of examples without any set of tools to deal with them.
\end{quote}

%
We import non-trivial constructions from model theory to build interesting linear spaces. Since we are interested in finding theories rather than structures, we construct families of similar Steiner systems that are similar both combinatorially and model theoretically. Perhaps this technique could become a tool in
studying infinite linear spaces as stability has already influenced graph
theory \cite{MalShreg}.

Studying the $(a,b)$-cycle graph \cite{CameronWebb} associated with Steiner triple systems (Definition~\ref{abgraph}), already yielded a perspicuous proof that we have constructed continuum many theories (Corollary~\ref{contmu}). In \cite{BaldwinsmssII} we extend the notion of graph cycle from Steiner triple systems to Steiner $q$-systems, for $q$ a prime power, and
produce examples of $T_\mu$ that have only finite cycles (called paths in the more general situation) in the prime model but infinite cycles in all others. 
 By cutting away
from the class of finite models some that have low $\delta$-rank, it is fairly
easy to guarantee that all models of $T_\mu$ are $2$-transitive.  By making
a relatively large such cut Hrushovski \cite[Example 5.2]{Hrustrongmin} produced
an example, which as a side effect, is a Steiner triple system. But this construction does not generalize uniformly, as ours does, to get Steiner
$k$-systems for larger $k$. With less extreme surgery we find in \cite{BaldwinsmssII} theories of $q$-Steiner system such that every model is $2$-transitive and thus the path graph is uniform in a sense inspired by \cite{CameronWebb}.

There have been a number of papers that use model theoretic techniques and,
in at least one case, the Hrushovski construction, to investigate linear spaces and Steiner systems. Our approach differs by invoking a predimension function
inspired by Mason's $\alpha$-function, and focusing on the combinatorial
{\em consequences of strong minimality} by investigating  the family of similar
(elementarily equivalent) structures of arbitrary cardinality arising from a
 particular strongly minimal $k$-Steiner system.  
  In contrast, Evans \cite{Evanssteiner} constructs
 Steiner triple systems using a variant of the Hrushovski construction without discussing their
  stability class.  At the opposite end of the stability spectra from
  our result, Barbina and Casanovas \cite{BarbinaCasa} find existentially
  closed Steiner triple systems that are $TP2$ and $NSOP_1$ by a traditional Fra{\" \i}ss\'{e} construction. Remark~\ref{onbc} compares their example with ours in more detail.
   Between these
  extremes, Hytinnen and Paolini \cite{HyttinenPaolinifree} show that the Hall construction of free projective planes yields a strictly stable theory.   Conant and Kruckman \cite{ConKr} find  an existentially closed
projective plane and prove it is $NSOP_1$ but not simple.  Their construction involves a generalized Fra{\" \i}ss\'{e} construction for the existential completeness as well as the Hall construction.

 Thus, there are four techniques that construct infinite linear spaces in a range of stability classes:  taking all extensions in a given universal class but insisting on finite amalgamation  in a standard Fra{\" \i}ss\'{e} construction \cite{BarbinaCasa}, building one chain of models
  carefully \cite{HyttinenPaolinifree}, combining these two methods but allowing the amalgam of finite structures to be countable \cite{ConKr},   and, as here, restricting the amalgamation class to guarantee a
   well-behaved $\acl$-geometry.
%
%


Section~\ref{sm} provides background on strong minimality and linear spaces, and proves the bi-interpretablity between the one and two-sorted approach. Sections~\ref{context} and \ref{primgood} lay out the distinctions in the basic theory between the general Hrushovski approach and the specific dimension function for linear spaces studied here. In Section 5 we prove the main existence theorem for strongly minimal Steiner systems and in Section~\ref{fcon} we discuss  the connection with recent work on the model theory of Steiner systems and expound the underlying properties which show that our examples have the usual `geometric' properties of Hrushovski constructions. We thank the referee for a very helpful report.

\section{Strong Minimality, Linear Spaces, Matroids and planes}\label{smaetal}



	The goals of this paper and the sequel are to construct  strongly minimal linear spaces, in fact, Steiner systems and to investigate some of the relevant connections between model theory and combinatorics.
%
	 In this section we describe strong minimality on the one hand, and the combinatorial notions of linear space, matroid, and some notions from design theory on the other.  The sophisticated study of strongly minimal sets depends on the general framework of one-sorted first-order logic; linear systems are usually studied in a two-sorted first-order logic, while matroids are rarely formalized (See Section~\ref{sec_matroids}.).
 We  explore here the role of and translations between these various  `formalisms'.   Most of our work takes place in the following context:


\begin{definition}[Linear Spaces in $\tau$]\label{taulin} Let $\tau$ contain a single ternary relation symbol $R$ which holds of  sets  of $3$ distinct elements in any order. $\bK^*$, the class of {\em linear  spaces},  consists  of the $\tau$-structures that satisfy:
any two distinct points determine a unique line when $R$ is interpreted as collinearity. That is, $R(x,y,z) \wedge R(x,y,w) \rightarrow R(x,w,z)$.
Each pair of elements is  regarded as lying on a (trivial) line; each non-trivial line is a maximal $R$-clique.
\newline
$\bK^*_0$ denotes the collection of finite structures in $\bK^*$.
\end{definition}

The switch from a $2$-sorted to a $1$-sorted formalism leads to some peculiar notation. In the two-sorted world, a line in $(M; P^M, L^M)$ can gain points when $M$ is extended.  In the one-sorted context a  line is a subset of the universe which is definable from any two points lying on it. But this definition is non-uniform. If the line is trivial (only two points) the definition is $x=a \vee x = b$; if the line is non-trivial  the definition is $R(a,b,x)$.
As a model $M$ is extended, not only may a line gain points, but the correct such definition can change.

\subsection{Strongly Minimal Theories}\label{sm}

A complete theory $T$ is strongly minimal if every model of $T$ is a strongly
minimal structure.   Prototypic examples of strongly minimal theories include
completions  of the pure theory of equality, vector spaces, and algebraically
 closed fields.

We define the {\em model theoretic  algebraic closure} of a set $A \subseteq M$ to be\footnote{$(\exists^{k!} x) \varphi(x,\overline{a})$ means that there are exactly
$k$-solutions of $\varphi(x,\overline{a})$; we similarly use $(\exists {}^{>k} x)$. These
 are abbreviations of first-order formulas.}:
$$\mathrm{acl}_M(A) =\{ b \in M : M \models \varphi(b,\overline{a}) \, \wedge \, \text{ for some $k$ } (\exists^{k!} x) \varphi(x,\overline{a})\},$$
where the $\varphi(x,\overline{a})$ vary over all formulas with parameters from $A$.
In any  strongly minimal structure $M$, the operator $\mathrm{acl}$ induces a matroid (pre-geometry) on
the subsets of $M$ (see e.g. \cite{BaldwinLachlansm}). This pre-geometry is infinite dimensional if $M$ is saturated.
  If $\mathrm{acl}_M(a) = \{a\}$, for every $a \in M$, then $(M, \mathrm{acl}_M)$ is a simple matroid (a combinatorial geometry).
Strong minimality imposes significant restrictions on the structure $M$ due to the following:

\begin{fact} \label{eliminf} If $M$ is strongly minimal, then for every formula
$\varphi(x,\overline{y})$, there is an integer $k = k_\varphi$ such that for any
 $\overline{a} \in M$, $(\exists {}^{>k_\varphi} x) \varphi(x,\overline{a})$ implies that there are
 infinitely many solutions of $\varphi(x,\overline{a})$, and thus finitely many solutions of $\neg\varphi(x,\overline{a})$.
 \end{fact}

 This is an easy consequence of the compactness theorem: if the conclusion fails the collection of
 sentences $\{(\exists {}^{>k} x) \varphi(x,\overline{y}) \wedge  (\exists {}^{>k} x) \neg\varphi(x,\overline{y})\colon k< \omega\} $
 is finitely satisfiable and so
 realized by some $\overline{a}^*$ in an elementary extension $N$ of $M$, which contradicts {\em strong minimality}. This result allows us, by suppressing the dependence of $k$ on $\varphi$, to introduce the abbreviation $(\exists {}^{\infty} x) \varphi(x,\overline{a})$ for
$(\exists {}^{>k_\varphi} x) \varphi(x,{\overline{a}})$.
As, in our context, the second assertion implies the first, which is usually not first-order.

 Fact~\ref{eliminf} has an immediate consequence for any strongly minimal linear space, $(M,R)\in \bK^*$ (cf. Definition~\ref{taulin}),
 where all lines have at least 3 points:
 there can be no infinite lines. Suppose $\ell$ is an infinite line. Choose $A$ not on $\ell$. For each $B_i, B_j$ on $\ell$ the lines $AB_i$  and $AB_j$ intersect only in $A$. But each line $B_i$ has a point not on $\ell$ and not equal to $A$. Thus
 $\ell$ has an infinite definable complement, contradicting strong minimality.
More strongly, we observe:

\begin{fact}\label{bndlen} If $(M,R)$ is a strongly minimal linear space, then there exists an integer $k$ such that all lines have length at most $k$.
\end{fact}
   As, $R(x,y,z)$ means\footnote{We require any triple satisfying $R$ to be of distinct points.} $x,y,z$ are collinear, i.e.\! $x$ is on the
  line determined by $y,z$, applying Fact~\ref{eliminf} we see that there is
  $k = k_R$ such that $(\exists {}^{>k_R} x) R(x,a,b )$ implies the
  line through $a,b$ is infinite, which contradicts the preceding paragraph.  In particular, there can be no strongly minimal affine or projective plane, since in such planes the number  points on a line
  must equal the number of lines through a point ($+1$ in the finite affine case).
%



%



\subsection{(Families of) Linear Spaces} \label{linsp}

We begin with the  notion of linear space as expounded in \cite{BatBe}. We formalize this notion in the usual first-order two-sorted way. In  Definition~\ref{taulin} we provided a one-sorted formalization of linear spaces, and in Theorem~\ref{bint} we will prove that the two definitions are bi-interpretable.


\begin{definition}[Linear Spaces in $\tau^+$]\label{linspace}
A {\em linear space} is a structure $S$ for a vocabulary $\tau^+$ with unary predicates $P$ (points)  and $L$ (lines) and a binary relation $I$ (incidence) satisfying the following properties:
\begin{enumerate}[(A)]
\item any two distinct points lie on at exactly one line;
\item each line contains at least two points.
\end{enumerate}
$\bK^+$ denotes the collection of $\tau^+$-structures that are linear spaces.
\end{definition}


\begin{remark} We omit in Definition~\ref{linspace} the usual non-trivality condition that there are at least three points not on a common line. It will of course be true of the infinite structures that we construct, but allowing even the empty structure is
technically convenient.
\end{remark}



%
%


 While \cite{BatBe} deals almost exclusively with {\em finite} linear spaces, the definition extends (as the authors noted) to allow infinite spaces.
We pause to describe several different descriptions of linear spaces, most notably {\em pairwise balanced designs} (as defined in \cite{WilsonI}):
%

\begin{definition}[PBD]\label{pbddef} A finite design is a pair $(X,\Lscr)$ where $X$ is a finite set and $\Lscr$ is a family $\{B_i: i\in I\}$ of (not necessarily distinct) subsets (blocks) of $X$.
\begin{enumerate}[(1)]
\item For $v\geq 0$ and $\lambda > 0$ integers, and $K$ a set of positive integers, a design $(X,\Lscr)$ is a $(v,K,\lambda)$-PBD, {\em Pairwise Balanced  Design}, if and only if:
\begin{enumerate}[(a)]
\item $|X| = v$;
\item $|B_i| \in K$;
\item\label{(iii)} every two element subset of $X$ is contained in exactly $\lambda$ blocks $B_i$.
    \end{enumerate}
    \item A Pairwise Balanced Design is said to be a Steiner system if $\lambda = 1$ and $|K| = 1$ (i.e.\! all blocks have the same size).

            \end{enumerate}

    An infinite PBD is obtained by omitting the requirement that $v$ is finite.
   \end{definition}

%
%
 If $K = \{k\}$, we adopt the standard notation of
        Steiner $k$-system.

   Condition (1)(c) is read as asserting that the design is {\em pairwise balanced} with index $\lambda$.
    Any  finite linear space is a $(v,K,\lambda)$-PBD for some $K$ and with $\lambda = 1$.
    Fact~\ref{bndlen} gives (1) of the next theorem; (2) is a consequence of our main construction.


    \begin{theorem}\label{designcon} \begin{enumerate}[(1)]
    \item
    A strongly minimal infinite linear space in the vocabulary $\tau$ (cf. Definition~\ref{taulin}) is a $(v,K,1)$-PBD for some finite set of integers $K$.
    \item For each $3 \leq k < \omega$, we construct continuum-many strongly minimal infinite linear spaces in the vocabulary $\tau$ that are Steiner $k$-systems.

        \end{enumerate}
        \end{theorem}



%

\subsection{One and Two-Sorted Formalization}\label{12sort}

We explore the historical connections between the one and two-sorted approach to combinatorial geometry and indicate that while our formalizations are bi-\hspace{0.0000000001cm}interpretable in the usual sense of model theory they differ in important ways. In particular, as mentioned in Section~\ref{sm}, the one-sorted version can be strongly minimal while the two-sorted one cannot.
Hilbert's axiomatization of geometry is naturally formulated as a first-order two-sorted incidence geometry\footnote{Although he includes two non-first-order axioms; all the properly geometric work is first-order axiomatized \cite{BaldwinconI, BaldwinconII}.} and this framework is developed  in, e.g., Hall \cite{Hall_proj}. This tradition is continued with Definition~\ref{linspace} of linear spaces as two-sorted structures in a vocabulary $\tau^+$ for first-order logic.
  Tarski aimed for a first-order foundation for Euclidean geometry and pioneered a single-sorted approach to geometry summarised in \cite{GivantTarski}.  Here the fundamental relation is a ternary predicate
  interpreted as `betweenness' or more generally as `collinearity'. In order to apply standard model theoretic tools, we provide a first-order single-sorted framework in a vocabulary $\tau$ that is equivalent (for our purposes; recall, however, that Morley rank is not preserved) \mbox{to the study of linear spaces.}


In the next definitions, we regard a  linear space in the vocabulary $\tau^+$ (cf. Definition~\ref{linspace}) as a $\tau$-structure (cf. Definition~\ref{taulin});
 this is easily done.  Given a $\tau^+$-structure $B$ as in Definition~\ref{linspace}, define a $\tau$-structure $A$ by letting $A$ be $P(B)$, the points of
$B$, and defining $R(a,b,c)$ if and only there is line $\ell$ in $B$ such \mbox{that each of
$a,b,c$ is on $\ell$.}

\begin{remark}\label{def_plane}
{\rm We now show that the class $\bK^*$ (Definition~\ref{taulin}) of single-sorted  linear spaces is bi-interpretable with
 the class $\bK^+$ of linear spaces in the two-sorted vocabulary $\tau^+$ (cf. Definition~\ref{linspace}).
  Notice that conditions (A) and (B) of  Definition~\ref{linspace} imply that every pair of distinct lines intersects in at most
  one point.
Also, recall that we allow models with  no points or  lines.}
\end{remark}



%





%

We define a pair of mutually inverse bijections from  the models of a class of $\tau$-structures to a class of $\tau^+$-structures and back that are uniformly definable, respect isomorphism, and preserve substructure.
 The notion that `bi-interpretability' means `same' requires some clarification.  On the one hand, we have already mentioned that the transformation here does not preserve Morley rank/degree.  This is because the lines of the $\tau^+$ structure are interpreted as imaginary elements (equivalence classes) of  the associated $\tau$-structure (More concretely; this is a $2$-dimensional interpretation \cite[212]{Hodgesbook}.).
On the other hand, such properties as decidability, $\aleph_1$-categoricity, and $\lambda$-stability are preserved by  first-order bi-interpretability.

 While the next theorem explicitly gives
 an isomorphism of categories (with embeddings as  morphisms), by changing notation we could construct a bi-interpretation in the classical sense of  \cite[Section 5.3]{Hodgesbook} between $\bK^+$ and $\bK^*$. For example, the domain of the interpretation of $\bK^+$ into $\bK^*$ in part (1), which Hodges would label $\partial_F$, is:  $\Delta(A^2) \cup (A^2- \Delta(A^2))/E$. Our formulation is awkward for the usual applications to decidability but natural for our `equivalence' between structures. Such a reformulation is a real strengthening since  bi-interpretability  of $A$ and $B$ is equivalent to their endomorphism rings being continuously isomorphic \cite{AhlbrandtZiegler} while mere isomorphism of  those monoids gives equivalent categories of models as in \cite{Lascarcat}. But 
 \cite{Evansetal} shows that there are $\aleph_0$-categorical structures which have isomorphic but not continuously isomorphic endomorphism monoids.

%

%
%
%
%
%
%
%

%

	\begin{theorem}\label{bint}
	\begin{enumerate}[(1)]
	\item
There is an interpretation $F$ of $\mathbf{K}^+$ into $\mathbf{K}^*$.
That is, for every $A \in \mathbf{K}^*$  there
  is a  $\tau^+$-structure  $F(A) \in \bK^+$ definable without parameters in $A$.
  %
	\item
There is an interpretation $G$ of  $\mathbf{K}^*$ into  $\mathbf{K}^+$.
That is, for every $B\in \mathbf{K}^+$  there
 is a  $\tau$-structure  $G(B) \in \bK^*$ definable without parameters in $B$.

\item For any $A\in \bK^*$, $G(F(A))$ is definably isomorphic to $A$ and for any $B\in\bK^+$, $F(G(B))$ is definably isomorphic to $B$.  Thus we have a bi-interpretation.


\end{enumerate}


\end{theorem}

	\begin{proof}  We prove (1). Let $A \in \mathbf{K}^*$.
Set $P = \{ (a, a) \colon  a \in A \}$ as the set of points of the $\tau^+$-structure $F(A)$. Towards describing the lines, define the following equivalence relation $E$ on $A^2 - P$ by declaring $(a, b) E (c, d)$ if and only if the following condition is met:
	\begin{equation}\tag{$\star$}
 \{ a, b \} = \{ c, d \} \text{ or } \{ a, b \} \cup \{ c, d \} \text{ is an $R$-clique}.
\end{equation}
%
%
%
We verify that $E$ is  transitive. To this end, suppose that
$(a, b) E (c, d)$ and $(c, d) E (e, f)$,  $e \neq f$,
 $\{ a, b \} \neq \{ c, d \}$ and $\{ c, d \} \neq \{ e, f \}$.
  Since each pair is of distinct elements  both $\{a,b,c,d\}$ and $\{c,d,e,f\}$ are $R$-cliques and since two points determine a line
  $\{a,b,c,d,e,f\}$ is an $R$-clique and transitivity is established.
 Now, let
$$L = \{ [(a, b )]_E : (a, b) \in A^2 \text{ such that } a \neq b \}$$
be the set of lines of $F(A)$. For $(p,p) \in P$ and $[(a, b)]_E \in L$ define the following point-line incidence relation:
$$(p,p) I [(a, b)]_E \Leftrightarrow \exists (c, d) \in [(a, b)]_E \text{ such that } p \in \{ c, d \}.$$
Clearly, $F(A)$ is definable in the $\tau$-structure $(A,R)$.  We show that $F(A)\in\bK^+$, i.e.\! Definition \ref{linspace} is satisfied.
 Obviously, Axiom (B) is satisfied.  We prove axiom (A). Towards this goal, let $\ell_1$ and $\ell_2$ be two distinct lines of $F(A)$  that intersect (via the definition of $I$) in two distinct points $(b_1,b_1)$ and $(b_2,b_2)$.
By hypothesis $\ell_1 \neq \ell_2$ and so, we can assume $\ell_1 = [(b_1,b_2)]_E$ and there is $(c, d) \in A^2$ such that $c \neq d$, $\neg E((b_1, b_2),  (c, d))$ and $(c, d) \in \ell_2$.
Note that any $E$-equivalence class of element with more than 3 elements consists of an $R$-clique and distinct $R$-cliques can intersect in only one point; so, we finish.

\smallskip
\noindent We prove (2). Let $B\in \bK^+$.  Define  the $\tau$-structure $G(B) =(A,R)$ by letting $A$ be the points of
$B$ and defining $R(a,b,c)$ if and only if $a,b,c$ are distinct and there is
a line $\ell$ in $B$ such that each of
$a,b,c$ is on $\ell$.  Since $B$ is a linear space the axioms of $\bK^*$ are immediate.

\smallskip
\noindent We prove (3) by showing that up to definable isomorphism $G$ is $F^{-1}$.  Fix $A$ and $F(A)$ from (1).
We analyze the composition $G(F(A))$ and show the image is definably
isomorphic to $A$. The set of points, $P^{F(A)}$, is the diagonal
$\Delta(A^2)$ of $A^2$. Map $(a,a)$ to $a$. The set of lines  of $F(A)$ is
$L^{F(A)}=(A^2- \Delta(A^2))/E$.  Let $m \in L^{F(A)}$ and suppose
 $ (a_0,a_0), (a_1,a_1), (a_2,a_2)$ are on $m$, where the $a_i$ are distinct.
   By the definition of $I$ in $F(A)$,  for each $i<3$ there exists an $a'_i$ such
that for $i\neq j$, $[(a_i,a'_i)]_E = [(a_j,a'_j)]_E$.
By $(*)$ this implies the $a_i, a'_i$ for $i< 3$ (some may be repeated) form an $R$-clique in $A$.  Thus $G(F(A))$  is definably isomorphic to $A$.
\newline Now we reverse the procedure and show that for $B \in \bK^+$, $F(G(B))$ is
 definably isomorphic to $B$. This is even easier. If $a,b,c$ are collinear in
 $B$, then $G(B) \models R(a,b,c)$ (Note $P^B$ is the domain of $G(B)$). For this, recall  the argument in part (1) showing $F(A) \in \bK^*$ takes
  collinear points of $A$ into a clique composed of elements of the diagonal
   of $G(B)$, which correspond to a clique in $B$.  Applying this argument to
   $G(B)$ completes the proof.
\newline Finally, this shows, in the case at hand, the essential point of \cite{Makgeom}, that $F$ is onto from $\bK^*$ to $\bK^+$.
\end{proof}

%

\subsection{Connections with Matroids}\label{sec_matroids}

The convention in matroid theory is
to regard the rank  as (normal geometrical dimension) + 1. For example, a `plane' is a rank~$3$ matroid.
 By a plane we here mean a model of a first-order single-sorted representation of the class of  simple matroids of rank~$3$.
  In describing this representation we lay out a formal correspondence (i.e.\!\ a bi-interpretation)
 between the matroidal and axiomatic approaches (as incidence structures) to geometry. The functorial correspondence between matroids and certain incidence structures is well-known to experts, but, at the best of our knowledge, the formal correspondence by model-theoretic means in Lemma~\ref{keyobs}  has, like that in Theorem~\ref{bint}, not been made explicit in the literature.


	As is well-known, see e.g. \cite[Chapter 2]{white}, matroids can be defined using many different notions as primitive. Among
 them are the notions of dependent set, circuit, independent set, basis, etc. In this work we will assume as primary the notion of {\em
 dependent set}.   
 In e.g. \cite{Oxley, white}, a collection $\Dscr$ of dependent sets is any collection of non-empty finite sets, closed under superset, and satisfying the well-known Exchange Axiom of Definition \ref{preaxiomatization}(\ref{exchange}). The matroid theorist writes axioms in the fashion of Euclid, Hilbert in 1899 \cite{Hilbertgeom}, or Bourbaki; there is no formal language. In fact, no
 standard logic can directly express these axioms, since the collection of dependent sets contains finite set of various cardinalities.
  Notionally, his arguments and definitions can be formalized in ZFC, but this is not an issue to him.  It is however crucial to our enterprize to describe our structures in first-order single-sorted logic.

 For this, as in \cite{Baldwindep}, we first work in
a relational vocabulary $\check \tau = \{ R_n : 1 \leq n < \omega \}$, where $R_n$ is an $n$-ary relation symbol.
Our axioms on $\check \tau$-structures  {\em first require} that each $R_k$ is a {\em uniform-$k$-hypergraph}, that is, $M \models R_n(a_1, ..., a_n)$ implies:
\begin{enumerate}[(1)]
\item $a_i \neq a_j$ for every $ 1 \leq i < j \leq n$;
\item $M \models R_n(a_{\sigma{(1)}}, ..., a_{\sigma{(n)}})$, for each $\sigma \in Sym(\{ 1, ..., n\})$.
\end{enumerate}

Consequently, if $X \subseteq M$, $(x_1,..., x_n)$ is an injective enumeration of $X$ and $M \models R_n(x_1, ..., x_n)$, then  we can write $M
\models R_n(X)$. Given a $\check \tau $-structure $M$ and $D
\subseteq_{\omega} M$ we say that a set $D$ is dependent if $M \models
R_{|D|}(D)$.   The  {\em further axioms}  in Definition~\ref{preaxiomatization} require  the $R_n$ to code in
 this way dependent sets of size $n$, for $1 \leq n < \omega$. For emphasis, we write the first and third axioms as $\check \tau$ sentences but we use the abbreviations introduced above to make the exchange axiom easier to read.

 \begin{definition}[Planes in $\check \tau$]\label{preaxiomatization} Following \cite{white} (see, in particular, \cite[Proposition 2.2.3 and Theorem 2.2.6]{white}) the class $\mathbf{K}^{\check \tau}$ of simple matroids of rank $\leq 3$ can be defined as the class of $\check \tau $-structures $M$ such that each $R_k$ is a uniform-$k$-hypergraph and satisfy the following further axioms:
\begin{enumerate}[(1)]
\item \label{small} $(\forall x)\neg R_1(x)$, $(\forall x,y)[ x \neq y \rightarrow \neg R_2(x,y)]$;
\item \label{exchange} if $D_1, D_2 \subseteq_{\omega} M$ are dependent and $D_1 \cap D_2$ is not dependent, then for every $a \in M$ we have that $D_1 \cup D_2 - \{ a \}$ is dependent;
\item  for all $n\geq 4$, $\forall x_1, \ldots x_n R_n(x_1, \ldots x_n)$.

\end{enumerate}
We call $\bK^{\check \tau}$ the class of planes.
\end{definition}

In matroid parlance, condition (1) asserts  that  we consider only {\em simple}
 matroids; in the language of combinatorial geometry it asserts that the structure
is a geometry, not merely a pre-geometry.
The more usual requirement for  a matroid that a superset of a dependent set is dependent
follows immediately from (1) and (3). When dealing with simple matroids of rank $3$ we can replace $\check \tau$ by the vocabulary with a single ternary relation symbol $R = R_3$  (see Definition~\ref{taulin} and Theorem~\ref{bint}).

The formulation in the last paragraphs deliberately smudges the transition between the informal and the formal first-order viewpoint.  In the further development we will have to remind ourselves of the conditions we put on the $R_n$ for $n> 3$.


Table~\ref{tab2.4} may help in navigating among the various choices of vocabulary (language) for first-order axiomatizations of linear spaces and matroids.

 \begin{table}[h]
$$\begin{array}{|c|c|c|}
\hline
\text{Language} & \text{Class} & \text{Context} \\
\hline
	\tau & \bf K^* & \text{One-sorted linear spaces (cf. Definition~\ref{taulin})} \\
\hline
	\tau^+ & \bf K^+ & \text{Two-sorted linear spaces (cf. Definition~\ref{linspace})} \\
\hline
	\check \tau & \mathbf{K}^{\check \tau} & \text{Matroids of rank $3$ (cf. Definition~\ref{preaxiomatization})} \\
\hline
\tau & \check{\bf K} & \text{Matroids of rank $3$ as $\tau$-structures (cf. Definition~\ref{axiomatization})} \\
\hline
\end{array}
$$\caption{The various contexts/languages of Section~\ref{smaetal}.\label{tab2.4}
}
\end{table}

\begin{definition}[Planes in $\tau$] \label{axiomatization} Let $\tau$ contain a single ternary relation symbol $R$. And, let
$\psi_n(x_1, \ldots x_{n}) $ assert that the $x_i$ are distinct. $\check \bK$ is the class of $\tau$-structures that satisfy
Definition~\ref{preaxiomatization} when $x \neq x $ is substituted for $R_1(x)$, $x =y$  for $R_2(x,y)$, $R(x,y,z)$ for  $R_3(x,y,z)$, and
$\psi_n$ for $R_n$ when $n \geq 4$.  
\newline $\check \bK_0$ denotes the collection of finite structures in $\check \bK$.
\end{definition}

Definition~\ref{axiomatization} is motivated by the following result, showing that the objects we create are planes in the matroid sense (cf. Definition~\ref{preaxiomatization}).
%

\begin{lemma}\label{keyobs}
The axiom schema of Definition~\ref{axiomatization} determines a rank 3 matroid structure on each member of $\check \bK$.
\end{lemma}

\begin{proof} Let the $\tau$-structure $A$ satisfy the axioms from Definition~\ref{axiomatization} concerning $R_3$, in particular, the exchange axiom for $R_3$.
The only obstruction now is checking the exchange axiom under the hypothesis that every four or more element set is declared dependent in $A$. Suppose $D_1$ and $D_2$ are arbitrary sets with at least four elements. By the substitutions for $R_1$ and $R_2$, $|D_1 \cap D_2| \geq 3$ and
so $D_1 \cup D_2$ has at least
seven points and so is dependent. Thus $A$ is a simple rank 3 matroid.
\end{proof}

\begin{remark}\label{distinctions}{\rm
Lemma~\ref{keyobs} could be generalized to any $k$ using only $R_i$ for $3 \leq i \leq k < \omega$.  
Of course, if the space arises as e.g. $F_q^n$, the $n$-space over a $q$-element field,  and $k <n$ the matroid dependence by membership in $\check \bK$ will be stronger than the dependence relation arising from the native linear space.

\smallskip
\noindent
We distinguish among $\bK^*$, $\bK^{\check \tau}$ and $\check \bK$ in order to be able to  axiomatize certain notions in first-order logic.
We often say a structure `is' a matroid, meaning a matroid structure can be imposed. The notion of a matroid is a property expressed in ZFC.
But if we formalize the matroid or linear space notions in one-sorted
 first-order logic we must be more careful. A $\check \tau$-structure which belongs to $\bK^{\check \tau}$ is a matroid if it is a model of the axioms in
 Definition~\ref{preaxiomatization}.  A $\tau$-structure in $\bK^*$ may admit
  matroid structures of any finite rank.  But a $\tau$-structure in $\check \bK$ is a
   matroid of rank $3$ because it satisfies the sentences $\psi_n$ from
   Definition~\ref{axiomatization}.  We are pedantic about the $\psi_n$ in order to
    ensure that the structures at the end of our complicated construction are rank $3$ matroids and so `planes'.
    In view of
    Lemma~\ref{keyobs}, we can regard any linear system as a  rank $3$ matroid (and the limit structures to have any finite rank we please).
 That is, any such linear system admits the structure of a rank $k$ matroid for any finite $k$, and so there was no need to \mbox{restrict to the rank $3$ case.}}
    \end{remark}

%

\section{The Specific Context}\label{context}


	In this section we introduce the specific context in which we will work for the rest of the paper. The main component of this section is the introduction of a new predimension function $\delta$ (cf. Definition~\ref{defdelrank}), which will be the essential ingredient in the construction of our strongly minimal Steiner systems. This predimension function $\delta$ was introduced in \cite{Paoliniwstb} and it is inspired by Mason's $\alpha$-function, a well-known measure of complexity for matroids introduced by Mason in \cite{mason}. We will give an explicit definition of our function $\delta$ without introducing the matroid theoretic machinery needed to define the $\alpha$-function. For the reader interested in this connection we refer to \cite[Section~3]{Paoliniwstb}, where this is carefully explained.

\begin{notation}\label{basicnot}
\begin{enumerate}[(1)]
	\item\label{hat} For any class $\bL_0$ of finite structures for a vocabulary $\sigma$ that is closed under substructure, $\hat \bL_0$ denotes the class of all $\sigma$-structures $M$ such that every finite substructure of $M$ is in $\bL_0$.
	\item\label{K_0} Given an arbitrary class of structures $\bL$ for a vocabulary $\sigma$ we denote by $\bL_0$ the class of finite structures in $\bL$. (For convenience, we allow the empty structure.)
	\item We write $\backsimeq$ for isomorphism, $X\subseteq_\omega Y$ for finite subset, and if $B \subsetneq C$, we may write$\hat C$ for $C-B$.
\end{enumerate}
\end{notation}


We will define below  several classes of structures; in particular $\bK_0$,
$\bK_\mu$, and $\bK^{\mu}_{d}$ (see Table~\ref{tab_classes} for references).
Furthermore, there will be various (closure operators)/(dependence relations) on structures in each class.  Since
each one of them could naturally be called `geometric', we avoid this term and give them each a different tag (see also Table~\ref{tab_dep}):

\begin{notation}[Notions of Dependence]\label{indnot} Let $M \in \mathbf{K}^*$ (cf. Def. \ref{taulin}) and $A \subseteq M$.
\begin{enumerate}[(1)]
\item The intrinsic closure operator (cf. Definition~\ref{K0def}) is denoted by $\mathrm{icl}(A)$.
\item The $d$-closure operator (cf. Definition~\ref{defd-cl}) is denoted by $\ddcl(A)$.
\item The algebraic closure operator is denoted by $\mathrm{acl}(A)$. (We use the standard model theoretic notion for algebraic closure, i.e.\! $a \in \mathrm{acl}(B)$ means that $a$ is in a finite set definable with parameters from $B$.)
   \item The subspace closure $\mathrm{cl}_R(X)$ in $A$, the smallest subset $B$ of $A$ containing $X$  such that if $a\in A$ satsfies $R(b_1,b_2,a)$ with the $b_i \in B$, then $a \in B$.
\end{enumerate}
\end{notation}

A key fact, Lemma~\ref{dclinacl},  asserts that on a $d$-closed structure $M$ (cf. Definition \ref{defd-cl}) in the class $\bK_\mu$,  notions (3) and (4) are equivalent; this is central for proving strong minimality.

Tables~\ref{tab_dep} and \ref{tab_classes} fix the notation introduced in Definition~\ref{indnot} and the classes of models discussed at various places in the text.

 \begin{table}[h]
$$\begin{array}{|c|c|}
\hline
\text{Notation} & \text{Name} \\
\hline
	\mathrm{icl}(A) & \text{intrinsic closure} \\
\hline
	\mathrm{acl}(A) &  \text{algebraic closure} \\
\hline
	\ddcl(A)& \text{$d$-closure} \\
\hline
	\mathrm{cl}_R(X)& \text{subspace closure} \\
\hline
\end{array}
$$\caption{Notions of dependence.\label{tab_dep}
}
\end{table}

\begin{table}[h]
$$\begin{array}{|c|c|}
\hline
\text{Notation}     &  \text{References} \\
\hline
	 \mathbf{K}^*   &  \text{Definition~\ref{taulin}} \\
\hline
	 \mathbf{K}_0^* &  \text{Definitions~\ref{taulin} and \ref{basicnot}(\ref{K_0})} \\
\hline
	 \mathbf{K}_0   &  \text{Definition~\ref{K0def}} \\
\hline
	 \hat \bK_0     &  \text{Definitions~\ref{K0def} and \ref{basicnot}(\ref{hat})} \\
\hline
	 \bK_\mu        &  \text{Definition \ref{Kmu}(\ref{Kmuitem})} \\
\hline
	 \bK^{\mu}_{d}  &  \text{Definition \ref{defd-cl}(\ref{mu-d})} \\
\hline
\end{array}
$$\caption{The classes of structures relevant to our construction.\label{tab_classes}
}
\end{table}

The following notation will clarify the distinction between 2-element lines (a.k.a. trivial lines) which are understood to hold of arbitrary pairs of elements from  models in $\bK^*$ and lines where the relation
symbol $R$ is explicit (cf. Definition~\ref{taulin}).

\begin{definition}\label{def_rank} Let $A \in \mathbf{K}^*$ and $A \subseteq B$ with $B \in \mathbf{K}^*$ (cf. Definition \ref{taulin}).
\begin{enumerate}[(1)]
\item A {\em line} of $A$ is an $R$-closed subset $X$ of $A$ such that all the points from $X$ are collinear.  In particular, if two points $a \neq b \in A$ and there is no $c \in A$ with $R(a,b,c)$, then $\{a,b\}$ is a line.  We call such lines `trivial'.

\item We denote the cardinality of  a line  $\ell\subseteq A$ by $|\ell|$, and, for $B \subseteq A$, we denote by $|\ell|_B$ the cardinality of $\ell \cap B$.
   \item We say that a line $\ell$ contained in $A$ is {\em based in}  $B \subseteq A$ if $|\ell\cap B| \geq 2$, in this case we write $\ell \in L(B)$.
\item	\label{nullity} The {\em nullity of a line} $\ell$ contained in a structure $A \in \mathbf{K}^*$ is:
	$$\mathbf{n}_A(\ell) = |\ell| - 2.$$
\end{enumerate}
\end{definition}

Note that if $ B\subseteq A$ are both in $\mathbf{K}^*$,  and $\ell \subseteq A$ is a line then $\ell \cap B$ may be  in $L(B)$ (if it has at least two points) but may not be $R$-closed in $A$ (i.e.\! if $\ell -B \neq \emptyset$).

With these
notions in hand,  we introduce the new rank $\delta$ that is central to
  this paper\footnote{Mermelstein \cite{Mermelthesis} has independently studied variants on this rank, but only in the infinite rank case so the intricate analyis of primitives in this paper does not arise.}. It has two key features: (i) it is based on the notion of `dimension' of a line; (ii) the associated geometry is flat, and so we get counterexamples to Zilber's conjecture (see Section~\ref{Kmusec} for details.).

\begin{definition}\label{defdelrank} For $A \in \mathbf{K}^{*}_0$ (recall Definitions~\ref{taulin} and \ref{basicnot}(\ref{K_0})), let:
	$$\delta(A) = |A| - \sum_{\ell \in L(A)} \mathbf{n}_A(\ell).$$
\end{definition}

\begin{proposition}\label{comp_prop} Let $A$ and $B$ disjoint subsets of a structure $C \in \mathbf{K}^*_0$. Then:
	\begin{enumerate}[(1)]
	\item if $\ell \in L(AB)$ and $\ell \in L(B)$, then $\mathbf{n}_{AB}(\ell) - \mathbf{n}_{B}(\ell) = |\ell|_A$;
	\item $\delta(A/B) := \delta(AB) - \delta(B)$ is equal to:
	$$|A| - \displaystyle \sum_{\substack{\ell \in L(AB) \\ \ell \in L(A) \\ \ell \not\in L(B)}} \mathbf{n}_{AB}(\ell) - \sum_{\substack{\ell \in L(AB) \\ \ell \in L(A) \\ \ell \in L(B)}} |\ell|_A - \sum_{\substack{\ell \in L(AB) \\ \ell \not\in L(A) \\ \ell \in L(B) }} |\ell|_A.$$
\end{enumerate}
\end{proposition}

We rely on Proposition~3.6 and Lemma~3.7 of \cite{Paoliniwstb}, which assert (the content of Lemma~\ref{delta_lemma} is also known as ``submodularity'' of the $\delta$ function):

  \begin{definition}\label{K0def}
	\begin{enumerate}[(1)] \item  Let:
	$$\bK_0 = \{ A \in \mathbf{K}^{*}_0 \text{ such that for any } A' \subseteq A,\delta(A') \geq 0\}, $$
and $(\bK_0, \leq)$ be as in \cite[Definition 3.11]{BaldwinShiJapan}, i.e.\! we let $A \leq B$ if and only if:
	$$A \subseteq B \wedge \forall X (A \subseteq X \subseteq B \Rightarrow \delta(X) \geq \delta(A)).$$
\item We write $A < B$ to mean that $A \leq B$ and $A$ is a proper subset of $B$.
\item For any $X$, the least subset of $A$ containing $X$ that is strong in $A$ is called the {\em intrinsic or self-sufficient closure} of $X$ in $A$ and denoted by $\mathrm{icl}_A(X)$  or $\overline{X}$. 
\end{enumerate}
\end{definition}

Since in the current situation we are dealing with integer coefficients for $\delta$  the intrinsic closure of every finite set is finite.

%


 \begin{remark}{\rm Note that $\bK_0$ has many fewer structures that $\bK^*_0$. In particular, no projective plane (except the Fano plane, Example~\ref{fano}) or space $A$ over a finite field is in  $\bK_0$;
as, for each such $A$, $\delta(A)<0$.}
\end{remark}

 We give a general conceptual analysis for submodularity\footnote{This result is proved by computation in \cite{Paoliniwstb}.} and flatness of $\delta$ that  clarifies the proofs of Lemmas~ \ref{delta_lemma} and \ref{getflat} (flatness of $d$). 

%


\begin{definition}\label{allflatdef}
\begin{enumerate}

 \item \label{flatdelta}  For a sequence
 $F_1, . . . , F_s$ of elements of $\bK_0$. For
$\emptyset \subsetneq S \subseteq \{1, . . . , s\}$, we let $F_S = \bigcap_{i\in s} F_i$ and $F_{\emptyset} = \bigcup_{1\leq i \leq s} F_i$.
 We say that $f$ is {\em flat} if for all
such $F_1, . . . , F_s$ we have:
$$ (*) \hspace{.1in} f(\bigcup_{1\leq i \leq s} F_i) \leq \sum_{\emptyset \neq S} (-1)^{|S|+1} f(F_S).$$
\item\label{defflat}  Suppose $(A, \cl)$ is a pregeometry on a structure $M$ with dimension function
$d$ and $F_1, . . . , F_s$ are finite-dimensional $d$-closed subsets of $A$.  Then $(A, \cl)$ is {\em flat} if $d$ satisfies equation $(*)$.

\end{enumerate}
\end{definition}

%

In the basic Hrushovski case, $\delta$ is  flat because it is the difference between two functions, the cardinality of each set, which satisfies inclusion-exclusion, and  counting  the number of occurrences of $R$ in  each set, which undercounts.
We now note our $\delta$ is similarly represented and that $\delta$ is modular
on the appropriate notion of free amalgam: $A \oplus_C B$ in $\bK_0$.

\begin{definition}\label{defcanam} Let $A \cap B =C$ with $A,B,C \in \bK_0$. We define $D:= A \oplus_{C} B$ as follows:

\begin{enumerate}[(1)]
	\item the domain of $D$ is $A \cup B$;
	\item
a pair of points  $a \in A - C$  and $b \in B - C$ are on a non-trivial line $\ell'$ in $D$ if and only if there is a line $\ell$ based in $C$ such that $a \in \ell$ (in $A$) and $b \in \ell$ (in $B$). Thus, in this case, $\ell'=\ell$ (in $D$).
\end{enumerate}
\end{definition}


Lemma~\ref{delta_lemma}.3 does not follow from submodularity but depends on the particular choice of free amalgam which is driven by `two points determine a line'.

	\begin{lemma}\label{delta_lemma} 
\begin{enumerate}
\item $\delta$ is flat (Definition~\ref{allflatdef}~\ref{flatdelta}).

\item  Let $A, B, C \subseteq D \in \mathbf{K}^{*}_0$, with $A \cap C = B$. Then: $$\delta(A/B) \geq \delta(A/C),$$
which an easy calculation shows is equivalent to submodularity: $$\delta(A\cup C) = \delta(A) + \delta(C) - \delta(B).$$
 \item \label{canext} If $E\cap F = D$, $D\leq E$ and $E,F, D \in \bK_0$ then
$G = F \oplus_D  E$ is in $
\bK_0$.  Moreover, $\delta( A \oplus_{C} B)= \delta(A) +\delta(B) - \delta(C)$ and any $D$ with $C \subseteq D \subseteq A \oplus_{C} B$ is also free.  Thus, $F \leq G$.
\end{enumerate}
\end{lemma}

\begin{proof}  1) Recall $\delta(A) = |A| - \Sigma_{\ell \subseteq A} (|\ell| -2$). Observe that if $A,B$  are
sets and $\ell$ is a line in $A\cup B$, then:

$$|\ell|=|\ell \cap A| + |\ell \cap B| - |\ell \cap (A\cap B)|.$$
But in computing $\delta(\bigcup_{1\leq i \leq s} F_i)$
on the right hand of (*)
 one must sum  for each $S$ only over those lines based in $F_S$.  Thus for example, in the case of two sets $A,B$, if a line is based in $A-B$ and has a single point in $C-B$ (and none in $B$) that point will not be counted on the right-hand-side but will be on the left. So the subtracted term of $\delta(F_S)$ is under-counted and $\delta(F_S)$ is over-counted. This is not corrected at the next step because no $\ell$ is based there. Thus, $\delta$ is flat.


2) For such combinations of counting functions, submodularity is just the notion of flat for two sets.

3) We need to check that each pair of points $a_0,a_1$
determine a unique line in $G$. Without loss of generality, one is in $F-D$ and the other in $E$. Suppose for contradiction there are two distinct lines on which both of $a_0,a_1$ are incident. If both lines are contained in $F$, the claim is obvious. But, if not, Definition~\ref{defcanam} guarantees   that both of $a_0,a_1$  are on a unique line based in $D$.

By the general submodularity argument, $\delta( A \oplus_{C} B)\leq \delta(A) +\delta(B) - \delta(C)$. But the definition of the free amalgamation guarantees that each line that intersects $A-B$ and $C-B$
in based on two points in $B$. There is no undercount as  there may be in 2) so we have equality.
\end{proof}

Reference \cite{BaldwinShiJapan} provides a set of axioms for {\em strong substructure}.
 These axioms can be seen to hold in our situation using Lemma~\ref{delta_lemma}.

 \begin{fact}\label{conclusion_ax} $(\bK_0, \leq)$ satisfies Axiom A1-A6 from \cite[Axioms Group A]{BaldwinShiJapan}, i.e.\!:
\begin{enumerate}[(1)]
	\item if $A \in \bK_0$, then $A \leq A$;
	\item if $A \leq B \in \bK_0$, then $A$ is a substructure of $B$;
	\item if $A, B, C \in \bK_0$ and $A \leq B \leq C$, then $A \leq C$;
	\item if $A, B, C \in \bK_0$, $A \leq C$, $B$ is a substructure of $C$, and $A$ is a substructure of $B$, then $A \leq B$;
	\item $\emptyset \in \bK_0$ and $\emptyset \leq A$, for all $A \in \bK_0$;
	\item\label{item6} if $A, B, C \in \bK_0$, $A \leq B$, and $C$ is a substructure of $B$, then $A \cap C \leq C$.
\end{enumerate}
\end{fact}


	We use the following notion of genericity:

\begin{definition}\label{defgen}
The countable
model $M\in \hat \bK_0$ is {\em $({\bK_0}, \leq)$-generic} when:
\begin{enumerate}[(1)]
\item
     if $A\leq M, A\leq B\in \bK_{0}$, then there exists $B'\leq M $
such that $ B\backsimeq_{A} B'$;
\item
 $M$ is a union of finite substructures.
\end{enumerate}
\end{definition}

\section{Primitive Extensions and Good Pairs}\label{primgood}

Using only the $\delta$ function one can build up models in $\bK_0$ from well-defined building blocks: primitive extensions and good pairs (Definition~\ref{prealgebraic}). This section
is  an analysis of these foundations. In the next section we use them to study the complete theories we are constructing.

	\begin{definition}\label{prealgebraic} Let $A, B \in \bK_0$ with $A \cap B = \emptyset$ and $A \neq \emptyset$.
	\begin{enumerate}[(1)]

\item
We say that $A$ is a {\em  primitive extension} of $B$ if $B \leq A$ and there is no $A_0$ with $B \subsetneq A_0 \subsetneq A$ such that $B \leq A_0 \leq A$.
 Equivalently, we describe a primitive pair as $(B,A)$ where $B$ and $A$ are disjoint (and so $BA$ is the set in the initial description).
\item  If $\delta(A/B) =0$, we write $0$-primitive. We stress that in this definition while $B$ \mbox{may be empty, $A$ cannot be.}
	\item  We say that the $0$-primitive pair $A/B$ is {\em good} if there is no $B' \subsetneq B$ such that  $(A/B')$ is $0$-primitive.
When discussing good pairs, usually $A$ and $B$ are disjoint; for ease of notation, sometimes $A$ is confused with $A \cup B$.

	\item If $A$ is $0$-primitive over $B$ and $B' \subseteq B$ is such that we have that $A/B'$ is good, then we say that $B'$ is a {\em base} for $A$ (or sometimes for $AB$).  
	\item If the pair $A/B$ is good, then we also write $(B,A)$ is a {\em good pair}.
\item  We sometimes use the notation $\hat C$: if $B \subsetneq C$, then $\hat C = C - B$.
\end{enumerate}	
\end{definition}

	\begin{remark}\label{rmk_gp} {\rm Note that if $C$ is primitive over the empty set then the unique base for $C$ is $\emptyset$. For, if there is
$B\neq \emptyset$ with $B \subsetneq C$ with $C$ based on $B$, then $\emptyset \leq B$ and $B \subsetneq C$ contradicting that $C$ is primitive over the empty set.
\newline This does not forbid the existence of $C \in \bK_0$ such that $\delta(C/\emptyset) = 0$ but $C$ is not primitive over $\emptyset$; on this see Lemma~\ref{nlf}.}
\end{remark}


%


%
%

\begin{example}\label{fano}{\rm Some sets are based on the empty set.  In particular, if $C$ is the $\tau$-structure representing the unique 7 point projective plane (often called the Fano plane), then
$\delta(C)=
0$.  And it is easy to see $(\emptyset,C)$ is a good pair.}
\end{example}


	In earlier variants of the Hrushovski's construction one was able to prove the existence of a {\em unique} base $B'$ for {\em any} given $0$-primitive extension $A/B$. Unfortunately, this assertion is {\em false} in the current situation, cf. Example \ref{example_bases}.  We will make up for this with a careful examination of the structure of good pairs that almost regains uniqueness.

	\begin{example}\label{example_bases} {\rm For $A \in \bK_0$ containing $m + 2$ points $p_1, ..., p_{m+2}$ on a line $\ell$ and for some $c$ such that $c \not\in \{ p_1, ..., p_{m+2} \}$ but  $c$ is on $\ell$ in $A \cup \{c\}$; we have that $c$ is $0$-primitive over $A$, and any  pair of points in $\ell \cap A$ constitutes a base for $c/A$.}
 \end{example}

 The following preparatory results allow us to characterize primitive extensions and eventually prove amalgamation for $(\mathbf{K}_{\mu}, \leq)$ (cf. Conclusion~\ref{conclude}).

 	\begin{proposition}\label{oddpointout} Let $B \in \bK_0$ and $b \in B$ such that $b$ does not occur in any $R$-tuple from $B$, then $\delta(B) = \delta(B - \{ b \}) + 1$.
\end{proposition}

	\begin{proof} As $b$ is on no line based in $B - \{ b \}$ this follows \mbox{from Definitions \ref{def_rank} and \ref{defdelrank}.}
\end{proof}

Using the above proposition, we can see:
	\begin{proposition}\label{uniqueness_base} Let $A, B \in \bK_0$ with $A \cap B = \emptyset$, $AB \in \bK_0$ and $B \leq AB$. Then:
	\begin{enumerate}[(1)]
	\item if there exists $b \in B$ such that $b$ does not occur in any $R$-tuple from $AB$, and  $B'$ denotes $B - \{ b \}$, then $\delta(A/B) = \delta(A/B')$.

	\item if the $0$-primitive pair $A/B$ is good (cf. Definition \ref{prealgebraic}(2)), then for every $b \in B$ we have that $b$ occurs in an $R$-tuple from $AB$.

	\end{enumerate}
\end{proposition}

	\begin{proof} It suffices to prove (1), and (1) is clear  by applying Proposition~\ref{oddpointout} to $AB$  as follows:
	$$\delta(A/B) = \delta(AB) - \delta(B) = (\delta(AB') + 1) - (\delta(B') + 1) = \delta(AB') - \delta(B').$$
\end{proof}
%

We use the following technical lemma to prove Lemma~\ref{primchar}, which characterizes good pairs.

\begin{lemma} \label{techforZiegler} Suppose $C$ is a  primitive extension of $B$ such that $|(C-B)| \geq 2$,
  then every non-trivial line $\ell$ with $\ell \cap (C-B) \neq \emptyset$ intersects $B$ in at most one point. Furthermore, if $C$ is $0$-primitive, then any point in $(C-B)$ lies on two lines based in $(C-B)$.
\end{lemma}

\begin{proof} Let  $\ell$ be a line that intersects $(C-B)$. Then $\ell$ is not based in $B$ since, if so, for any $c\in \ell \cap (C-B)$, $Bc$ would contradict the primitivity of $C$.
 But then, if $C$ is $0$-primitive, any    $c\in (C-B)$ must lie on a line based in $(C-B)$, as otherwise, letting $C' = (C-B)-\{c\}$,  Proposition~\ref{oddpointout} implies $\delta(C'/B) =\delta(C/B)-1 = 0-1 < 0$,
contradicting $B \leq C$.  But, in fact,  $c\in (C-B)$ must lie on two lines based in $(C-B)$.  If it is based on only one, deleting $c$ decrements both the number of points and the sum of the nullities of lines based in $(C-B)$ by $1$. So $\delta(C'/B) = 0$, contradicting that $C$ is $0$-primitive over~$B$. \end{proof}

%
%
%


 The next lemma is the {\em fundamental} tool for our analysis of primitive extensions.

\begin{lemma} \label{primchar}
 Let $B \leq C \in \bK_0$ be a primitive extension. Then there are two cases:
	\begin{enumerate}[(1)]
	\item $\delta(C/B) = 1$ and $C = B \cup \{ c \}$;
	\item $\delta(C/B) = 0$.
\begin{enumerate}[(2.1)]
\item
There is $c \in (C-B)$ incident with a line $\ell$ based in $B$ if and only if $|(C-B)| = 1$. In that case, any $B' \subseteq B$ with $B'\subseteq \ell$ and such that $|B'|=2$ yields a good pair $(B',c)$.  Furthermore, $c$ is in the relation $R$ with an element $b \in B$ if and only if $b$ is on the  unique line based in $B'$.
\end{enumerate}
\begin{enumerate}[(2.2)]
\item  If $|(C-B)| \geq 2$ then there is a unique base $B_0$ in $B$ for $C$.  Moreover, suppose $b \in B$ and $c \in (C-B)$. If $b$ and $c$ lie on a nontrivial line, then $b\in B_0$.  And every $b\in B_0$ lies on such a line, which must be based in $(C-B)$.
 \end{enumerate}
	\end{enumerate}
\end{lemma}

\begin{proof} We follow the case distinction of the statement of the lemma:
\newline {\bf Case 1}. Suppose $\delta(C/B) >0$ and there are distinct elements in $(C-B)$  that are not on lines based in $B$, then any one of them gives a proper intermediate strong
extension of $B$ that is strong in $C$. Thus $C$ must add only one element to $B$ yielding Case 1.
\newline {\bf Case 2}. Suppose $\delta(C/B) =0$.
\newline {\bf Case 2.1}. Suppose there is an element $c \in (C-B)$ which is on a line with two points in $B$, say $b_1,b_2$, and $|(C-B)| \geq 2$. Then clearly $Bc$ is a primitive extension of $B$ and $Bc \lneq BC$. Thus, $(C-B)$ must  be $
\{c\}$. Furthermore, $(\{b_1,b_2\},c)$ is a good pair.  So $C$ is based on $\{b_1,b_2\}$ and for any $b \in B$, $b$ is $R$-related to $c$ if and if
$R(b_1,b_2,b)$; otherwise $c$ would be on two lines based in $B$ (contradicting $B \leq C$). Conversely, if $|(C-B)| = 1$ then $c$ must be on a line based in $B$ since $\delta(C/B) = 0$.
\newline{\bf Case 2.2} $|(C-B)| \geq 2$ and $\delta(C/B) =0$.
%
\newline By Lemma~\ref{techforZiegler}, each line $\ell\in L((C-B))$ intersects $B$ in at most one point $b_{\ell}$.
 If there is no such $b_\ell$, then there is no $R$-relation between $(C-B)$ and $B$, so by Proposition~\ref{uniqueness_base}(2), $B = \emptyset$ and $C$ is based on $\emptyset$. As argued in Remark~\ref{rmk_gp}, that base must be unique.
\newline  If there is such a $b_\ell$, let $B_0$ be the collection of all the $b_\ell$, $\ell\in L((C-B))$.  By Lemma~\ref{uniqueness_base}.(1), $\delta(C/B_0)=\delta(C/B)$, and so $(B_0,C)$ is
a good pair. Further $B_0$ is the unique base for $C$ as these are the only elements of $B$ on lines that intersect $(C-B)$.
\end{proof}

 Omer Mermelstein provided us with  an example showing there
	 are infinitely many
primitives based on a single three element set.
 But the study of $(a,b)$ cycles
in \cite{BaldwinsmssII} led to stronger and simpler examples over smaller base sets.
Recall that any linear space with $3$-point lines is an example of Steiner triple system (i.e.\! in Definition~\ref{pbddef} we have $K = \{ 3\}$). The following definition will be used to prove Lemma~\ref{2prim}.


\begin{definition}[\!\!\hspace{-0.0000075cm}\cite{CameronWebb}]\label{abgraph} We define the notion of $(a,b)$-cycle graphs in Steiner triple systems.
 Fix any two points $a,b$ of a Steiner triple system $\Sscr = (P,L)$.  The cycle graph $G(a,b)$ has vertex set $P -\{a,b,c\}$ where $(a,b,c)$ is the unique block (Definition~\ref{pbddef}) containing the points $a$ and $b$. There is
an edge coloured $a$ (resp. $b$) joining $x$ to $y$ if and only if $axy$ is a block (resp. $bxy$ is a block) and the colors alternate.
\end{definition}


\begin{definition}\label{defcycle} Fix any two points $a,b$ of a Steiner $m$-system $\Sscr = (P,L)$.  We can build an $(a,b)$-cycle, $C_k$, $c_1, c_2, \ldots c_{4k}$ of length $4k$ by demanding $R(a,c_{2n+1}, c_{2n+2})$ for $0\leq n \leq 2k$,  $R(b, c_{2n+2},c_{2n+3})$ for $0\leq n < 2k$,  and $R(b, c_{1},c_{4k})$.
\end{definition}

In the Steiner triple system case a triple $a,b,c_1$ with $c_1$ not on $(a,b)$ determines a unique cycle as described in Definition~\ref{defcycle}. For $m$-Steiner systems with $m>3$, we can choose such cycles but not uniquely. Note that the lines determined by the pairs of points $c_n,c_{n+1}$  in Definition~\ref{defcycle} must be distinct.

%
%
%

\begin{lemma}\label{2prim}  There are infinitely many mutually non-embeddable primitives in $\bK_0$ over a two-element set.  In fact, there are infinitely many mutually non-embeddable primitives in $\bK_0$ over the \mbox{empty set and similarly over a $1$-element set.}
\end{lemma}

\begin{proof}  Over any $a,b$ for each $k$ build an $(a,b)$-cycle $C_k$ , as in Definition~\ref{defcycle}. $C_k$ has $4k$ points and $(\{a,b\} \cup C_k)\in \bK_0$ has $4k$ 3-element lines. So $\delta(\{a,b\}\cup C_k))= 2 = \delta(\{a,b\})$.
Primitivity easily follows since if the cycle is broken, the
$\delta$-rank goes up. So $(\{a,b\},C_k))$ is a good pair whose isomorphism type we denote by $\boldsymbol{\gamma}_k$.

\smallskip
\noindent
To get primitives over $\emptyset$, let $c$ be on $ab$ and add the relations
$R(c,c_1,c_{2k+1}$) and $R(c,c_{k+1},c_{3k+1})$.  Now the entire structure $D_k$ has $4k+3$ points and $4k+3$ lines and can easily be seen to be  $0$-primitive over the empty set. (Note that for $k=1$, this is another avatar of the Fano plane.)

\smallskip
\noindent
Now remove one of the last two instances of $R$ and the result \mbox{is primitive
over $a$ or $b$.}
\end{proof}


%

%

\section{The Class $\bK_\mu$}\label{Kmusec}


We now introduce the new classes of structures needed to obtain strong minimality.  Recall that we have two classes: (i) $\bK_0$ is a class of finite structures; (ii) $\hat \bK_0$ is the universal class generated by $\bK_0$. The new class $\bK_\mu \subseteq \bK_0$  adds additional restrictions so that the generic model for $\bK_\mu$
is a strongly minimal linear space, and, in fact, a Steiner $k$-system for some $k$.
Using Definition~\ref{ax}, we  axiomatize the subclass $\bK^\mu_d$ of $\hat{\bK}_\mu$ (the universal class generated by $\bK_\mu$) of those models that are elementarily equivalent to   the generic for $\bK_\mu$. We extend Table~\ref{tab_classes} to a Table~\ref{table4} including the new classes defined in this section.

\begin{table}[h]
$$\begin{array}{|c|c|}
\hline
\text{Notation}     &  \text{References} \\
\hline
	 \mathbf{K}^*   &  \text{Definition~\ref{taulin}} \\
\hline
	 \mathbf{K}_0^* &  \text{Definitions~\ref{taulin} and \ref{basicnot}(\ref{K_0})} \\
\hline
	 \mathbf{K}_0   &  \text{Definition~\ref{K0def}} \\
\hline
	 \hat \bK_0     &  \text{Definitions~\ref{K0def} and \ref{basicnot}(\ref{hat})} \\
\hline
	 \bK_\mu        &  \text{Definition \ref{Kmu}(\ref{Kmuitem})} \\
\hline
	\hat{\bK}_\mu        &  \text{Definition \ref{Kmu}(\ref{Kmuhatitem})} \\
\hline
	 \bK^{\mu}_{d}  &  \text{Definition \ref{defd-cl}(\ref{mu-d})} \\
\hline
\end{array}
$$\caption{The classes of structures relevant to our construction.\label{table4}
}
\end{table}

 The following notation singles out the effect of the fact that our rank
  depends on line length rather than the number of occurrences of a relation.

\begin{notation}[Line length]\label{linelength} We write $\boldsymbol{\alpha}$ for the isomorphism
type of the good pair $(\{b_1,b_2\},a)$ with
 $R(b_1,b_2,a)$. 
 \end{notation}


	\begin{definition} \label{Kmu}	Recall the characterization of primitive extensions from Lemma 4.8 of \cite{BaldwinPao}.
	\begin{enumerate}[(1)]
	\item\label{itemKmu} Let $\Uscr$ be the collection of functions $\mu$ assigning to every isomorphism type $\beta$ of a good pair $(B,C)$ in $\bK_0$ (we write $\mu(B,C)$ instead of $\mu((B,C))$):
	\begin{enumerate}[(i)]
	\item an integer $\mu(\beta) = \mu(B,C) \geq \delta(B)$, if $|C-B|\geq 2$;
	\item an integer  $\mu(\beta) \geq 1$, if $\beta = \boldsymbol{\alpha}$ (cf. Notation~\ref{linelength}).
	\end{enumerate}

\item
For any good pair $(B,C)$ with $B \subseteq M$ and $M \in \hat \bK_0$,
 $\chi_M(B,C)$ denotes the number of disjoint copies of $C$ over $B$ in $M$. Of course,  $\chi_M(B,C)$ may be $0$.

	\item\label{Kmuitem} Let $\bK_{\mu}$ be the class of structures $M$ in $ \bK_{0}$ such that if $(B,C)$ is a good pair, then  $\chi_M(B,C)  \leq \mu(B,C)$.
	\item\label{Kmuhatitem} $\hat \bK_\mu$ is the universal class generated by $\bK_\mu$ (cf. Notation ~\ref{basicnot}(\ref{hat})).
\end{enumerate}	
\end{definition}

In \cite{BaldwinsmssII}, we change the set $\Uscr$ in various ways (and explore the combinatorial consequences of this change in the resulting generic model). In this paper, we assume $\mu \in \Uscr$ unless specified otherwise.


The value of $\mu(\boldsymbol{\alpha})$ is a fundamental invariant
 in determining the possible complete theories of generic structures; in particular we will see that it determines the length of every line in the generic and thus in any model \mbox{elementary equivalent to it.}

%
%




\begin{remark}\label{primline}{\rm We analyze the structure of extensions governed by good pairs with isomorphism type $\boldsymbol{\alpha}$
from Notation~\ref{linelength}. 
 Suppose $\{b_1,b_2,a\}  \subseteq F \in \bK_{\mu}$ with $R(b_1, b_2, a)$.
The $0$-primitive extensions $C$  of $B= \{b_1,b_2\}$ with $|(C-B)|=1$ are exactly the points on the line $\ell$ through $b_1,b_2$.
Any pair of points $e_1, e_2$ from $F$ that are on $\ell$ form a base witnessed by   $(\{e_1,e_2\},a)$ with
 $R(e_1,e_2,a)\wedge R(b_1,b_2,a) $.

\smallskip
\noindent
   Most arguments for amalgamation in Hrushovski constructions (e.g. \cite{baldwin, Holland5, Hrustrongmin, Zieglersm}) depend on a careful
 analysis of the location of the {\em unique} base of a good pair. Here, when $|\hat{C}| =1$, the uniqueness disappears and one must focus on the line rather than a particular base for it.}
  \end{remark}


There are two general approaches to showing existence of complete
 strongly minimal theories by the Hrushovki construction.  One divides the construction into two
 pieces, free and collapsed \cite{Poizatgeneric,Zieglersm}.  The final theory
 is taken as the sentences true in the generic model.
  The second, as the original \cite{Hrustrongmin},
 provides a direct construction of the strongly minimal set.
We choose here to follow the Holland's version of this approach. She  insightfully emphasised  axiomatizing the theory of the  class $\bK^{\mu}_{d}$ of $d$-closed structures \cite{Holland5},which we now define, by clearly identifiable $\pi_2$-sentences. This established the model completeness which was left open in  \cite{Hrustrongmin}.  In fact, we axiomatize the theory $T_\mu$ of the class $\bK^{\mu}_{d}$, prove it is strongly minimal, and then observe that the generic satisfies $T_\mu$. 
%
%
%

\begin{definition}\label{defd-cl} Fix the class $(\bK_0,\leq)$ of $\tau$-structures as defined in Definition~\ref{K0def}.
\begin{enumerate}[(1)]
\item
For $A\in \hat{\bK}_0$,  $X \subseteq_\omega A$ and $a \in A$, we let:

$$d_A(
X) = \min\{\delta(Y): X \subseteq Y \subseteq_\omega A\},$$
and
$$d_A(a/X) = d_A(aX) -d_A(X).$$

\item\label{item_defcl} For $M\in \hat{\bK}_\mu$, and $X \subseteq_\omega M$: $$\ddclM(X) = \{a \in M: d_M(aX) = d_M(X)\}.$$
For infinite $X$, $a \in \ddclM(X)$ if $a \in \ddclM(X_0)$ for some $X_0\subseteq_\omega X$.


\item  For $M\in \hat{\bK}_\mu$ and $X \subseteq M$, $X$ is $d$-closed in $M$ if $d(a/X)= 0$
implies $a \in X$ (equivalently, for all $Y \subseteq_{\omega} M-X$, $d(Y/X) >0$).

\item\label{mu-d} Let $\bK^{\mu}_{d}$ consist
of those $M\in \hat{\bK}_\mu$ such  that $M \leq N$ and $N \in \hat \bK_\mu$ imply $M$ is $d$-closed in $N$.

    \end{enumerate}
\end{definition}

The switch from $\delta$ to $d$ is designed to ensure that $X \subseteq Y $ implies $d(X) \leq d(Y)$; the submodularity of $d$ is verified as in e.g. \cite{BaldwinShiJapan, Holland5, Zieglersm}, and so the function $d$ is truly a dimension function, thus inducing a matroid structure.

\begin{fact}\label{geomexists} The $d$-closure operator $\ddclM$ (cf. Definition~\ref{defd-cl}(\ref{item_defcl})) induces a combinatorial pregeometry on any $M \in \hat \bK_\mu$.
\end{fact}

%
%
%

%
%
%
%

We use good pairs to build our axiomatization,  $\Sigma_\mu$, of the theory of the class $\bK^{\mu}_{d}$.
 We write $\Sigma_\mu$  as the union of four sets of first-order $\tau$-sentences: $\Sigma^0_\mu$, $\Sigma^1_\mu$, $\Sigma^2_\mu$ and $\Sigma^3_\mu$.
Before listing them, we explain the origin of the third group: $\Sigma^2_\mu$. We would like
to just assert  the collection of {\em universal-existential} sentences: for all good pairs $(B,C)$ with $B
\subseteq M$, $\chi_M(B,C) = \mu(B,C)$.
Unfortunately, some  good
pairs may conflict with each others, and so, as far as we know,  the equality may fail for some good pairs when the base $B$ is not strong in the model.
%
%
Basically, this could happen because if  $(P,G)$ and $(Q,F)$ are good pairs with
$QF$ contained in $PG$ then realizing $(P,G)$ implies that $(Q,F)$ is automatically realized.  In particular, note that the $C$ of the good pair $(B,C)$ of Example~\ref{iclosedmatters} contains a new good pair $(B',C')$.

 The distinguishing  property of models $M \in \bK^{\mu}_{d}$ is that since every $0$-primitive extension over a finite {\em strong} subset of $M$ can be
embedded in $M$, by Lemma~\ref{getmax}, no proper $0$-primitive extension of $M$ is in $\hat{\bK}_\mu$.   In fact, this property characterizes the models that are elementarily equivalent to the generic.



Crucially, Holland\footnote{Holland provides a common framework for both {\em ab initio} constructions and fusions. The generality introduces considerations that are not relevant here, and our new predimension and the restriction to linear spaces introduce complications to her argument. Thus, for the convenience of the reader, we rephrased the argument for our situation.} expresses this failure  by a clearly motivated
$\pi_2$-sentence, which we expound in Remark~\ref{Hollsent}.
A salient point about the generic for $\bK_\mu$, denoted $\mathcal{G}_\mu$ (Notation~\ref{gennot}), is that
$\mathcal{G}_\mu \in \bK^{\mu}_{d}$.
 This fact is not used directly in the proof of strong minimality of $T_\mu$; we will observe it in Proposition~\ref{dcl}.

\medskip

%
%
%

One reason for the difficulty in the axiomatization is that the function $\mu$ is defined on arbitrary substructures, not strong substructures. {\em Restricting to strong substructure would inhibit if not prevent the $\pi_2$-axiomatization as  the strong substructure relation ($A \leq M$) is only type-definable.} Thus, in Lemma~\ref{isoqe}, we cannot assume $D$ is strong in both $E$ and $F$. In the following definition we rely on the terminology introduced in Definitions~\ref{prealgebraic} and \ref{Kmu}.

\begin{definition}\label{ax} $\Sigma_\mu$ is the union of the following four sets of sentences:
\begin{enumerate}[(1)]
\item $\Sigma^0_\mu$ is the collection of universal sentences axiomatizing $\bK_0$ as in Definition~\ref{K0def}.
   \item $\Sigma^1_\mu$ is the collection of {\em universal} sentences
that assert:
$$ B \subseteq M \;\; \Rightarrow \;\; \chi_M(B,C) \leq \mu(B,C).$$

    \item\label{item_sigma2} $\Sigma^2_\mu$ is a collection of {\em universal-existential} sentences $\psi_{B,C}$, depending on the good pair $(B,C)$, such   that for every occurrence of $B$ if $M \models \psi_{B,C} $ then for some good pair $(A,D)$ with $AD \subseteq BC$, any structure $N $ containing $MC$ satisfies  $\chi_N(A,D) > \mu(A,D)$ and so violates $\Sigma^1_\mu$.  See Lemma~\ref{proveax} for the explicit formulation of these sentences.
\item $\Sigma^3_\mu$ is the collection of existential sentences asserting that every line has length $\mu(\boldsymbol{\alpha}) +2$.

        \end{enumerate}
\end{definition}

%
%
%

          The  argument in Lemma~\ref{isoqe} that underlies both the axiomatization of $\bK^{\mu}_{d}$ and the amalgamation for $(\bK_\mu, \leq)$ differs from a mere amalgamation argument in  one significant way: $D \subseteq F$ but $D \leq F$
 is not assumed (on the other hand, $D \leq E$ {\em is} assumed).  We require several technical lemmas to address the difficulties arising from this fact.
 Our argument shows that if there is a model $M$ that
satisfies $\Sigma_\mu$, then we can find sentences to prevent
extensions in which $M$ is not $d$-closed.  The following example shows the
 necessity for the complications in proving Lemma~\ref{isoqe}: new primitives can occur in many ways.

 \begin{example}\label{iclosedmatters}{\rm
Construct the isomorphism type $\boldsymbol{\beta}$ of a good pair $(B,C)$ defined as follows.
  Let $B$ be two points $d_1,d_2$ and $C$ consists of six points $c_i$ for $i = 1, \dots 6$.
Let the non-trivial lines be $\{d_1, c_1,c_2, c_3\}$,$\{d_2, c_4,c_5, c_3\}$, $\{c_4,c_1,c_6\}$ and $\{c_5,c_2,c_6\}$.
So $C$ has 6 points and 4 lines each of nullity 1 so rank 2.
And $BC$ has 8 points and 4 lines, 2  of nullity 1 and 2 of nullity 2 so $BC$ also has rank 2. Check primitivity by inspection.


\smallskip
\noindent
Now turn this example on its head. Consider the following example of the setting of
Lemma~\ref{isoqe}. 
 Let $D = \{c_1,c_2\}$, 
   Let $F= D \cup \{c_3,c_4, c_5, c_6,d_2\}$,  
  and $E = D \cup \{d_1\}$. 
  $(D,E)$ is a good pair.  
   Amalgamating $F$ and $E$ over $D$
  we get a new
 realization $(B',C')$ of the {\em isomorphism type} $\boldsymbol{\beta}$ of the good pair $(B,C)$,
 which is not contained in either $D$ or $E$, but in $F \cup E$.
  This example does not violate
Lemma~\ref{isoqe} as $\mu(\boldsymbol{\alpha}) =2$ (and must be since there are $4$-element lines
in~$F$).}
 \end{example}

\begin{remark}\label{Hollsent}
 {\rm Example~\ref{iclosedmatters} shows that good pairs can conflict so we don't know  in general that a model $M$ of $T_\mu$ will satisfy $\chi_M(B,C) = \mu(B,C)$ for all good pairs $(B,C)$ that appear in $M$. We first prove in Lemma~\ref{isoqe}
  that each good pair
$(B,C)$  can only conflict with finitely many pairs $(B',C')$ and that that can happen only
if one pair is included in the other.
Following \cite{Holland5},
to guarantee  that $M
\in \bK^{\mu}_{d}$,  we assert by the formula $\psi_{B,C}$ (cf. Definition~\ref{ax}(\ref{item_sigma2})) that each  conflicting  pair $(A,D)$ is `almost realized' in $M$ so that adding points from $C$ contradicts $\Sigma^1_\mu$.}
\end{remark}

  Notice that in Lemma~\ref{isoqe} the fact that $(D,{E})$ is a good pair implies that $D \leq E$, and so we can use Lemma~\ref{canext} and thus consider $G = E \oplus_{{D}} F$. Notice that in Lemma~\ref{isoqe} the fact that $(D,{E})$ is a good pair implies that $D \leq E$, and so we can use Lemma~\ref{canext} and thus consider $G = E \oplus_{{D}} F$.
The following variant on \cite[Lemma 5.1]{Zieglersm} simplifies our original proof of Lemma~\ref{isoqe}.

\begin{lemma}\label{Zieg} Suppose $F \leq G$ and $F$ satisfies $\Sigma^0_\mu$. If there $C_i$ for $i<n$ that are pairwise disjoint over $B$ and the $(B,C_i)$ realize isomorphic good pairs.  Then at least one of the follows holds.
\begin{enumerate}
\item $B \subseteq F$
\item Some $C_i$ lies in $G-F$.
\end{enumerate}
\end{lemma}

\begin{proof}
Suppose $B \nsubseteqq F$ and $C_i\subseteq F$  for $1 \leq i \leq r -1$ and $C_i \cap F \neq \emptyset$ and $C_i \cap (E-F) \neq \emptyset$ if $r \leq i \leq r+s -1$.  Thus, $r+s =n$. Then, for each $i<r$, $B$ contains a point that is on a line that is based on  $C_i$ and none of the other $C_j$.  Thus $$\delta(B/F) \leq \delta(B/B\cap F) -r \leq \delta(B) -r.$$
But for $j<s$, $\delta(C_j/B\cup (C_j \cap F)) < 0$  by the definition of primitive; so $\delta(\bigcup_{j<s}C_{r+j} C_j / FB) < -s$. So
$$\delta(\bigcup_{j<s}C_{r+j} C_j / F) \leq \delta(B) - (r+s)$$ as required.
\end{proof}

\begin{lemma}\label{isoqe} Let $F, E \models \Sigma_\mu$,  $D \subseteq F$, and suppose that $(D,{E})$ is a good pair (and so in particular $D \leq E$). 
 Now, if $G = E \oplus_{{D}} F$ and for some good pair $(B,C) \subseteq G$
  we have $\chi_G(B,C) > \mu(B,C)$, then:
\begin{enumerate}[(A)]
\item 
if $|C| = 1$, $C = \{c\}$ and $c$ is on a line based on some  $B'\subseteq D$;
    \item if $|C| \geq 2$ then $B\subseteq E$ and there exists $C'$ with $BC' \backsimeq BC$,
   with   $C' \subseteq (E-D)$.
Further, if $D\leq F$, there is a copy $C''$ of $C$ over $B$ with $C''  =(E-D)$, and $B \subseteq D$.
\end{enumerate}
\end{lemma}

	\begin{proof}
 Since $G = E \oplus_{{D}} F$ we can use the notation and results of \ref{defcanam} and Lemma~\ref{delta_lemma}.~\ref{canext}.
 Note that $F, D, E$ are in $\hat \bK_\mu$ by the definition of the axioms $\Sigma_\mu$. Furthermore, $D \leq E$ and $E \in \bK_\mu$, by the definition of good pair.
 Let $\Cscr$ be a set of $\mu(B,C)+1$ disjoint copies of $C$ over $B$ in $G$, and list $\Cscr$ as $(C_1, ..., C_m)$, for $m = \mu(B,C)+1$.

 \noindent {\bf Case A}. $|C| = 1$.
\newline Then $(B,C)$ witnesses the isomorphism type $\boldsymbol{\alpha}$ from
  Definition~\ref{linelength}. So, there must be a line $\ell$ of size $\mu(B,C)+3$
 in $G$.
  Since $E$ and $F$ satisfy $\Sigma^1_\mu$, there must
   be $d \in F-D$ and $c \in E-D$ that lie on $\ell$.
  By Definition~\ref{defcanam}(2) of free almalgam $\ell$ must contain two points (say, comprising  $B'$) in
 ${D}$ that are connected to  $c \in E-D$.  Since $\{c\}$ is then primitive over $D$,  $E-D $ must be
  $ \{c\}$.  We finish the first claim. Note  $\chi_F(B',C)=\mu(B,C)$ as $\ell$ has $\mu({\boldsymbol{\alpha}})+2$ points in $F$.

 \noindent {\bf Case B}. $|C| \geq 2$.

 \begin{claim}\label{intoE} If $C \subseteq E-D$ is good over $B \subseteq F$, then $B \subseteq E$.
 \end{claim}

 \begin{proof} We show $B \subseteq E$. If not, there is a $b_1 \in B \cap (F-E)$  and since $C_j
 \subseteq (E-D)$ a line from $b_1$ to some $c\in C_j$.  Thus $c$ is on a line based on $D$ and so $C_j = E-D = \{c\}$. This contradicts $|C| \geq 2$
 so  $B \subseteq E$. \end{proof}

 We split into two cases depending on Lemma~\ref{Zieg}

 \noindent {\bf Case B.1}. Suppose $B \subseteq F$.

\noindent Since $\chi_F (B,C) \leq\mu(B,C)$, there must be a $C_i \in \Cscr$ that intersects $G-F = E-D$.
So, since $F\leq G$ and $C/B$ is primitive, $C_i \subseteq  G-F = E-D$.  But, since $E$ is primitive over $D$, $FE$ is primitive over $F$, so
$C_i = E-D$. By Case 2.2 of Lemma~\ref{primchar}, 
 $B$ is the only  subset of $F$ on which $C_i$ is based.
   Hence, as $BC_i \subseteq E$, we finish Case B.1 without using the supplemental hypothesis for the `further' in Case (B).

\smallskip

%

\smallskip
 \noindent {\bf Case  B.2}.  Suppose $B \not \subseteq F$.
By Lemma~\ref{Zieg}, we have the main claim; some $C_j$ lies in $E-D$. We prove the further.
 There must be a $C'\in \Cscr$ that intersects $F-D$, since $E \in \bK_\mu$.
  But $C'$ cannot split over $E$ since, $B \subseteq E$ by Claim~\ref{intoE}. As we now assume $D\leq F$, $E \leq G$; so $C' \subseteq (F-D)$. But then $C'$ is based on  a unique $B' \subseteq D$ since $D\leq F$. So $B = B'\subseteq D$.  But then $C_j$  is primitive over $D$ and based on $B\subseteq D$, and so $C_j = E-D$. Hence, $C_j$ is the required $C''$.
  This concludes the proof of Lemma~\ref{isoqe}.
\end{proof}

   The argument  for Lemma~\ref{getmap} differs from the standard only in requiring a special case for extending a line.

\begin{lemma}\label{getmap}
Suppose $A$ and $A'$ are  primitive over $Y$ with $\delta(A/Y) = \delta(A'/Y) = 0$ and both are based
on $B \subseteq Y$ with isomorphic good pairs $(B,\hat A)$ and  $(B,\hat A')$, where $\hat A = A-Y$ and $\hat A' = A'-Y$.  Then the map fixing $Y$ and taking $A$ to $A'$ is an isomorphism.
\end{lemma}

\begin{proof}
There are two cases depending on the cardinality of $\hat A$.
\newline {\bf Case 1}. $|\hat A| = 1$.
\newline As in case 1.1 of Lemma~\ref{primchar} let $\ell$ be a line which is based in $Y$ and suppose $\hat A = \{a\}, \hat A' =\{b\}$ are each on $\ell$ but neither is in $Y$.  Then, since both $a$ and $b$ are $R$-related only to the points on $\ell$  the map fixing $Y$ and taking $a$ to $b$ is an isomorphism.
\newline {\bf Case 2}. $|\hat A| \geq 2$.
\newline Applying Lemma~\ref{primchar}(2.2), there is a unique base $B$ (the $B_0$ of the lemma) and there is  a bijection $f$ between $|\hat A|$ and $|\hat A'|$ such that for each $b \in B$, $R(c_1,c_2,b)$ if and only $R(f(c_1),f(c_2),b)$. The union of that map with the identity on $Y$ is as required.
\end{proof}

We now show that any element of $\hat \bK_\mu$ (not just $\bK_\mu$) can be amalgamated (possibly with identifications) over a (necessarily finite) {\em strong} substructure $D$ of $F$ with a strong extension of $D$ to a member  $E$  of $\bK_\mu$. 

  \begin{conclusion}\label{conclude}  If $D \leq F \in \hat \bK_\mu$ and $D \leq E \in \bK_\mu$ then there is $G\in \hat \bK_\mu$ that embeds (possibly with identifications) both $F$ and $E$ over $D$.  Moreover, if $F \in \bK^\mu_d$, then $F = G$. In particular, $(\bK_\mu, \leq)$ has the amalgamation property, and there is a generic structure $\mathcal{G}_\mu \in \hat{\bf K}_\mu$ for $(\bK_\mu, \leq)$.
%

  \end{conclusion}

  \begin{proof}  Let $D,E,F $ satisfy the hypotheses. Clearly, we can assume that $D \leq E$ is a primitive extension.
   If $\delta(E/D) =k > 0$,
   Lemma~\ref{primchar} implies $k = 1$ and $E-D = \{a\}$. Now the disjoint amalgamation $E\oplus_D F$ is in $\hat \bK_\mu$ since $a$  is not $R$-related to any other element.
So, we are reduced to $0$-extensions and can refine the induction to
   assume $(D,E)$ is a good pair.
 %
We have an amalgam $G \in \hat \bK_0$ such that $G = E \oplus_D F$, $F\leq G$, and $E\leq G$. 
If $G \in \bK_\mu$, we finish.  If not,
there is an isomorphism type $\beta$ of a good pair $(B,C)$ and $(C_i: i < m)$ with $(B,C_i)\subseteq G$ realizing $\beta$ and such that $m > \mu(B,C)$. We now make a case distinction and show that in both cases we can embed $E$ into $F$ over $D$.
\newline {\bf Case 1}. $|C| =1$.
\newline By  Lemma~\ref{isoqe}(A) and by primitivity of $D \leq E$, we
 have that $|E-D| = 1$. But $E \in \bK_\mu$, and so $\chi_D(\boldsymbol{\alpha}) <
  \mu(\boldsymbol{\alpha})$, from which it follows that the element of $E-D$  can be embedded in $F-D$ over $B$.
\newline {\bf Case 2}. $|C| > 1$.
\newline The `further' clause of Lemma~\ref{isoqe}(B) shows that there must be a copy of $C$ equal to $E-D$.
Thus, using Lemma~\ref{getmap} and the argument in the last paragraph of the proof of Lemma~\ref{isoqe} (there is a copy of $C$ in $F-D$), we can conclude that we can embed $E$ into $F$ over $D$.

\smallskip
\noindent
For the `moreover', note that $M \in \bK^\mu_d$ implies that every proper extension $N$ of $M$ with $N \in \hat \bK^\mu_d$ satisfies $d(N/M) > 0$.
\end{proof}

\begin{corollary}\label{getmax} If $M \in \hat \bK^\mu_d$ and $B \leq M$, then for any
good pair $(B,C)$ with $C\cap M =B$, we have:
$$\chi_M(B,C) =\mu(B,C).$$
\end{corollary}

\begin{proof} By  Conclusion~\ref{conclude}, since $B\leq M$, there is an amalgamation in $\hat \bK_\mu$ of
$C$ and $M$ over $D$. But, $M$ and $MC$ cannot be freely amalgamated over $B$. As, in a putative amalgam $N$, $d_N(C/M)=0$. Whence since $M$ is $d$-closed, $C\subseteq M$, contradicting free amalgamation. By the `further' of Lemma~\ref{isoqe}(B).2, the violation of $\Sigma_1$ is given by the new copy of the pair $(B,C)$, and so $\chi_M(B,C) =\mu(B,C)$.
   \end{proof}

   \begin{question}  Is $D\leq M$ essential for the conclusion of Lemma~\ref{getmax}? {\rm  A complicated example showing necessity of this hypothesis in the fusion case (and thus the gap in Poizat's `proof' of existence of a
    Morley rank 2 expansion of a plane by a unary predicate) appears in \cite[\S 4]{BaldwinHolland}. The proof of Corollary~\ref{getmax} relies on that assumption  both in using the
   `further' of Lemma~\ref{isoqe} and  Conclusion~\ref{conclude}. Looking carefully at the proof of Lemma~\ref{isoqe} reveals that if there is a counterexample $(D,E)$, the failure is witnessed by a $(B,C)$ with $m = \mu(B,C) +1$ such that $B \subseteq E$, $B\nsubseteq F$, no $C_i \subseteq F$ and some $C_i \subseteq E-D$. Thus, $|E-D| \geq |C| + m$. It is unclear whether $B$ might be contained in $D-E$. Thus, we need something far different from Example~\ref{iclosedmatters} where we showed new isomorphism types of good pairs could appear in an amalgam but $|E-D|=1$.
   }
   \end{question}

%

\begin{notation}\label{gennot} Let $\Gscr_\mu$ denote the generic for $(\bK_{\mu} \leq)$ (cf. Conclusion~\ref{conclude}).
\end{notation}

Notice that it follows from Corollary~\ref{conclude} that every member of $\bK_\mu$ is strongly embeddable in $\Gscr_\mu$.

\begin{definition}\label{def_rich} Let $(\mathbf{K}_0, \leq)$ be as in the
context of Fact~\ref{conclusion_ax}. The structure $M$ is rich for the class
 $(\hat {\mathbf{K}_0}, \leq)$ (or $(\hat {\mathbf{K}_0}, \leq)$-rich) if for any
  finite $A,B \in \mathbf{K}_0$ with  $A \leq M$ and $A\leq B$ there is a
  strong embedding of $B$ into $M$ over $A$.
\end{definition}

%
%
%
%
%


Clearly, a generic is rich.  Even more,
since the definition of $\bK^{\mu}_{d}$ requires the embedding only of finite
extensions with dimension $0$, we have:

\begin{proposition}\label{dcl} Every rich model, and so in particular $\mathcal{G}_\mu$, is in $\bK^{\mu}_{d}$.
\end{proposition}

\begin{proof} We show that
every $(\bK_\mu,\leq)$-rich model $M$  is in $\bK^{\mu}_{d}$. Suppose for contradiction that there is an  $N \in \hat \bK_\mu$  with
$M \leq N$ and there is a $C \subseteq (N-M)$ such that $C$ is $0$-primitive over $M$.
By Lemma~\ref{primchar}, $C$ is based on some finite $B \subseteq M$. Since
$M \leq N$, $C$ is also primitive over $B_0 = \mathrm{icl}_M(B)$.
Since $M$ is rich there is a copy $C_1\subseteq M$ of $C$ over $B_0$. Now
 let $B_1 = \mathrm{icl}_M(C_1)$. Applying richness again we can choose another embedding  $C_2$  of $C$ into $M$ over $B_1$.
Continuing in this fashion, after less than $\mu(B,C) + 1$ steps we have contradicted $M \in \hat \bK_\mu$.
\end{proof}

%

%

	Now we explain the interaction between  the axioms $\Sigma^1_\mu$ and  $\Sigma^2_\mu$.  No extension of a model of $\Sigma^2_\mu$ by a good pair is in $\hat \bK_\mu$. This will yield the axiomatization of the theory of the $d$-closed
structures and thus of the generic (by Proposition~\ref{dcl}).

\begin{lemma}\label{proveax} The family of first-order sentences
 $\Sigma_\mu$ (Definition~\ref{ax}) defines the class of $d$-closed models.
\end{lemma}

\begin{proof} We use the notation of Lemma~\ref{isoqe}.
For $M \in \hat{\bK}_\mu$,
we say $M \oplus_D E$ is {\em bad} if for some good pair $(B,C)$ with $BC \subseteq DE$, $\chi_{M \oplus_D E}(B,C) > \mu(B,C)$.

 We first define for each good pair $(D,E)$ the formula $\psi_{(D,E)}$ described  in Definition~\ref{ax}. For each  duo of good pairs $(D,E)$ and $(B,C)$ with $BC \subseteq DE$ 
define the formula $\varphi_{(D,E),(B,C)}$ as follows.
Fix a model $M_0 \in \hat{\bK}_\mu$; choose a copy of $D\subseteq M_0$
such that $M_0 \oplus_D E$ is bad witnessed by $(B,C)$.
If $|C| >1$ choose by Lemma~\ref{isoqe}.B
$C_1,\ldots C_r$ (where $r =\mu(B,C)+1$)  that are disjoint copies of $C$
 over $B$ contained in $M_0 \oplus_D E$ and let $\overline{s}$ enumerate $H =(\bigcup_i C_i)-D) \cap M_0$.
   Let
  $\chi(\overline{v},\overline{x})$ be a possible
   atomic diagram of $H \cup D \subseteq M$, where $\lg(\overline{v}) = \lg(\overline{s})$,  for pairs $(M_0,D)$ as $M_0$ varies over $\hat{\bK}_\mu$ and $D$ varies over possible embeddings into $M_0$.
 Let \begin{equation}\varphi_{(D,E),(B,C)}:
   \bigvee_i (\exists \overline{v})\chi_i(\overline{v},\overline{x})\label{aformula}\end{equation} where the  $\chi_i$ are the finitely many possible such diagrams $\chi$.  We have shown that for any $M \in \hat{\bK}_\mu$ if $M \oplus_D E$ is a {\em bad} extension witnessed by $(B,C)$ then $M\models \psi_{(D,E)(B,C)}$.

    Let $\rho(\overline{x})$ be the atomic
   diagram of $D$.
Now we define $\Sigma^2_\mu$ and $\Sigma^3_\mu$ to assert a) each line has cardinality $\mu(\boldsymbol{\alpha})+2$  and b) each of the following (countable) collection of sentences (for all good pairs $(D,E)$), where $\rho(\overline{x})$ is the atomic
   diagram of $D$.
\begin{equation}\psi_{(D,E)}: \hskip .1in (\forall \overline{x})  [\rho(\overline{x}) \rightarrow \bigvee_{BC \subseteq DE}\varphi_{(D,E),(B,C)}(\overline{x})]. \label{keyformula}\end{equation}

 Now, if $M\models \Sigma_\mu$ then $M$ is $d$-closed. Since if not, there is an $N \in \hat{\bK}_\mu$ for some $(D,E)$, $M \oplus_D E \subseteq N$.
 If $|E|=1$ then condition a) is violated. Suppose $M\models \psi_{(D,E)}$ witnessed by $(B,C)$. If $|C| =1$ condition a) is again violated by Lemma~\ref{isoqe}.A.

 But, if $|C| >1$ some $\chi_i$ from Equation~\ref{aformula} will be satisfied in $M$. And, by Definition~\ref{delta_lemma}.~\ref{defcanam}, $\chi(\overline{v},\overline{x}) \cup \diag_{\rm qf}(E) \models \diag_{\rm qf}(HE)$ where $H$ is, as before, the interpretation of $\overline{v}$. This implies $\chi_{M \oplus_D E}(B,C)\geq \chi_{HE}(B,C)  > \mu(B,C)$ and we finish.
 \end{proof}


%
%

Recall (Definition~\ref{defd-cl}) that a finite set $X$ is $d$-independent when each $x \not\in \ddcl(X -\{x\})$, i.e.\! $d(X) > d(X -\{x\})$ for each $x \in X$.  It is then easy to establish the first of the following assertions by induction and the others follow.

\begin{lemma}\label{basicd} Let $M \in \hat{\bK}_\mu$ and let $Y$ be $d$-independent in $M$. For every finite $X \subseteq Y$ we have:
\begin{enumerate}[(i)]
\item $d(X) = |X|$;
\item $X\leq M$, and so $\mathrm{icl}_M(X) =X$;
\item there are no $R$-relations among elements of $X$.
\end{enumerate}
\end{lemma}

Using Lemmas~\ref{basicd} and \ref{conclude},
we follow Holland's proof  showing that $\Sigma_\mu$ axiomatizes the complete theory of $\bK^{\mu}_d$ (See Lemma 23 of \cite{Holland5}.).

%

\begin{lemma}\label{axcomp} If $M \models \Sigma_\mu$ then $M \in \bK^{\mu}_{d}$. Moreover, $\Sigma_\mu$ is an axiomatization of the complete theory $T_\mu$ of the class $\bK^\mu_d$.
\end{lemma}

\begin{proof}
By Lemma~\ref{proveax}, it suffices to show $\bK^\mu_d$ is $\kappa$-categorical
for $\kappa> \aleph_0$.
 Suppose now that $M, M'\models \Sigma_\mu$. By $\Sigma^3_\mu$ and by taking elementary extensions we
may assume that both have cardinality $\kappa > \aleph_0$. We show that $\Sigma_\mu$ is $\kappa$-categorical and so complete.
As geometries,
$M$ and $M'$ have bases $X,X'$ of the same cardinality.
By Lemma~\ref{basicd}, $X \leq M$ and $X'\leq M'$ and they are isomorphic by any bijection $f$.
  The  isomorphism  $f$ extends to one between $M$ and $M'$ since $M$ and
   $M'$ are built from
   $X$ and $X'$ by
    a sequence of
     $0$-primitive extensions and each step can be extended by
      Lemma~\ref{conclude}.
      \end{proof}




Having followed the outline of her proof, we have the analog to Holland's result \cite{Holland5} that the   strongly minimal Hrushovski constructions are model complete.

\begin{remark} Since the axioms $\Sigma_\mu$ are universal-existential and $T_\mu$  is $\aleph_1$-\hspace{0.0000000001cm}categorical, it is model complete by Lindstroms's `little theorem': that $\pi_2$-axiomatizable theories that are categorical in some infinite power are model complete \cite{Lindstrom}.
\end{remark}

 Our theories $T_\mu$ uniformize the result that there are only finitely many finite line lengths in any strongly minimal linear space (cf. Fact~\ref{bndlen}). We show in Corollary~\ref{contmu} using Lemma~\ref{2prim} that there are continuum-many strongly minimal theories $T_\mu$ such that in each of them all lines have fixed length $\mu(\boldsymbol{\alpha} +2$.

\begin{corollary}\label{contmu} There are continuum-many $\mu \in \Uscr$ (cf. Definition~\ref{Kmu}(\ref{itemKmu})) which give distinct first-order theories of Steiner systems. That is, there is $\mathcal{V} \subseteq \Uscr$ such that $|\mathcal{V}| = 2^{\aleph_0}$ and $\mu \neq \nu \in \mathcal{V}$ implies that $Th(G_\mu) \neq Th(G_{\nu})$ (recall Notation~\ref{gennot}).
\end{corollary}

\begin{proof} For any $X \subseteq \omega$, let $\mu_X$ assert that $\mu(\boldsymbol{\gamma}_k)$ (from the proof of Lemma~\ref{2prim}) is $3$ if $k\in X$ and $2$ if not (recall that it must be at least $2$).
Then, if $k\in X\setminus Y$, then $T_{\mu_X}
\not \equiv  T_{\mu_Y}$ (cf. Notation~\ref{gennot}), since there are three extensions in the isomorphism type $\mu(\boldsymbol{\gamma}_k)$
of some pairs $\{a,b\}$ in models of $T_{\mu_X}$ but not in models of $T_{\mu_Y}$.\end{proof}

%

%
%

\begin{lemma} \label{dclinacl} If $M \in \bK^{\mu}_{d}$, then for every $X \subseteq M$, $\ddcl(X) = \mathrm{acl}_M(X)$.  Thus, $T_\mu$ is strongly minimal.
\end{lemma}

\begin{proof} We first show that for $M \in \hat \bK_\mu$ the left hand side is contained in the right. If $Y$ is a finite subset of $M$,$\delta(Y/X) =0$, $Y$ is a union of a finite chain with length $k< \omega$ of extensions by good pairs $(B_i,C_i)$;
each is realized by at most $\mu(B_i,C_i)$ copies, and so:
$$|Y|\leq \sum_{i<k}\ \mu(B_i,C_i) \times |C_i|.$$
Thus, $Y \subseteq \mathrm{acl}_M(X)$.

\smallskip
\noindent
Concerning the other containment, let $M \in  \bK^\mu_d$, $a \in M$ and $X \subseteq_\omega M$. If $d(a/X) > 0$ and $X_0$ is a maximal $d$-independent subset of $X$, then $X_0 \cup \{a\}$ extends to a $d$-basis for $M$. Furthermore, in the proof of Lemma~\ref{axcomp}, we observed that any permutation of a $d$-basis extends
to an automorphism of $M$. Thus, if $a\notin \ddcl(X)$ , then $a \notin \mathrm{acl}_M(X)$. Hence, $\ddcl(X) = \mathrm{acl}_M(X)$, as desired.

\smallskip
\noindent
 Strong minimality follows, since for any finite $A$ there is a unique non-algebraic $1$-type over $A$, namely the type $p$ of a point $a$ such that: (i) $a$ is  not on any line based in $A$ (and so $\delta(a/A) =1$); (ii) $Aa$ is strong in any model. Clause (ii) is given by the collection of universal sentences forbidding any $B \supseteq Aa$ with $\delta(B) <\delta(Aa)$. Thus, in $\Gscr_\mu$ we have that $d(a/A) = 1$ for any $a$ realizing $p$.  Hence, any two realizations $a$ and $b$ of $p$ are such that $Aa \leq \Gscr$ and $Ab \leq \Gscr$, and thus they are automorphic by the genericity of $\Gscr_\mu$ (cf. Conclusion~\ref{conclude}). Hence, $p$ is a complete type.
 \end{proof}




\begin{notation}\label{fanoset}
 Let $F$ be the Fano plane and $\Fscr$ be the set of $\mu \in \mathcal{U}$ such that:
$$\mu(\emptyset, F) > 0.$$
\end{notation}

Lemma~\ref{nlf} shows that  for any $\mu \in \Fscr$ and $M \models T_\mu$, we have that $\mathrm{acl}_M(\emptyset)$ is infinite; by Ryll-Nardjewski, $T_\mu$ is not $\aleph_0$-categorical. In view of Lemma~\ref{dclinacl}, the countable models correspond exactly to the models of dimension $\alpha$ for $\alpha \leq \aleph_0$.


\begin{lemma}\label{nlf} Let $\mu \in \Fscr$. Neither the generic, $\Gscr_\mu$, nor any model of $T_\mu$  is locally finite with respect to $\ddcl = \mathrm{acl}$ (cf. Lemma~\ref{dclinacl}). Thus, $T_\mu$ is not $\aleph_0$-categorical and has $\aleph_0$ countable models.
Since the generic has infinite dimension, it is $\omega$-saturated.
\end{lemma}

\begin{proof} We show that the algebraic closure of the empty set is infinite. Construct a sequence $( A_i : i<\omega)$ in $\Gscr_\mu$ by
letting $A_0$ to be the Fano plane, which (Example~\ref{fano}) is easily seen to be $0$-  primitive over the empty set. Notice that there can only be finitely many realizations of the Fano plane in any model of $T_\mu$, and so $A_0$ is in the algebraic closure of the empty set. Now let $a_0,b_0,c_0$ be the vertices of the triangle in the standard picture of the Fano plane. Choose $a_1,b_1,c_1$ disjoint from $A_0$ so that
 $(a_0,a_1,c_1)$, $(b_0,b_1,c_1)$,  and $(a_1,b_1,c_0)$ are triples of collinear points. Then, letting $A_1 = \{a_0,b_0,c_0,a_1,b_1,c_1 \}$, it is to see that $A_1$ is a primitive
 extension of $A_0$.  Now  build $A_2$ by taking $a_1,b_1,c_1$ as the base and adding $a_2,b_2,c_2$ as in the construction of $A_1$ from $A_0$; and then iterate.
 Each stage (and hence the union)
can be strongly embedded as $A'_i$ in the generic. But then $\delta(A'_{i+1}/A'_i) = d(A_{i+1}/A_i) = 0$. By transitivity, with $A_\omega$ denoting $\bigcup_{i < \omega} A_i$, we have that for any finite $X\subseteq A_\omega$,
$d(X/A_0) = 0$. Since $\ddcl = \mathrm{acl}$ (Lemma~\ref{dclinacl}), we finish.
We constructed this sequence in the algebraic closure of the empty set, and so
it occurs in the prime model of $T$.
Thus, $\mathrm{acl}_M(\emptyset)$ is infinite for any model $M$ of $T_\mu$. By Ryll-Nardjewski, $T_\mu$ is not $\aleph_0$-categorical.
In view of Lemma~\ref{dclinacl}, as in any strongly minimal theory, these models correspond exactly to models of dimension $\alpha$ for $\alpha \leq \aleph_0$.
\end{proof}

\section{Further Context}\label{fcon}

In this section we place our work in the context of further work on the model theory of Steiner systems/linear spaces, studies on the consequences of flat geometries, and Hrushovski constructions.

\begin{remark}\label{onbc} {\rm We compare our examples with the construction in \cite{BarbinaCasa}  of  structures existentially closed for the class of all Steiner quasigroups.  Note that Steiner quasigroups are the quasigroups associated with  Steiner {\em triple} systems in \cite{BarbinaCasa}.}
\begin{enumerate}[(i)]
{\rm \item Their generic, denoted $\mathbb{M}_{\rm sq}$, has continuum many types over the empty set,  satisfies $TP_2$ and $NSOP_1$, and it is locally finite (but not uniformly locally finite) as a quasigroup. If $\mu \in \Uscr$, then it is obvious that $T_\mu$ fails the first three of these properties since it is strongly minimal.
Furthermore, we showed in Lemma~\ref{nlf} that our examples with $\mu \in \Fscr$ are not locally finite for $\mathrm{acl} =\ddcl$.
Strikingly, in $\mathbb{M}_{\rm sq}$, the definable closure is equal to the
 algebraic closure ($\dcl = \mathrm{acl}$). In  \cite{BaldwinsmssII} we show that this equality fails drastically in any $T_\mu$ with $\mu \in \Uscr$. }
{\rm \item The structure $\mathbb{M}_{\rm sq}$  is the prime model of its theory; our $\Gscr_\mu$ is saturated.
While the example in \cite{BarbinaCasa} is quantifier eliminable, ours is only model complete. The first is the model completion of the universal theory of
Steiner quasigroups.  Since each $M \in \bK_\mu$ can be extended to $N \in \bK^\mu_d$, the second is the model completion of the universal theory of $\hat \bK_\mu$ for the relevant $\mu$. Quantifier elimination does not follow since, despite the limited amalgamation in Conclusion~\ref{conclude},
$\hat \bK_\mu$ does not have amalgamation. }
{\rm \item In the introduction we mentioned further results on the combinatorics of strongly minimal Steiner systems and strongly minimal quasigroups from \cite{BaldwinsmssII}, and compared our approach with that of \cite{BarbinaCasa, ConKr, HyttinenPaolinifree}. }
\end{enumerate}
\end{remark}

We isolate for strongly minimal sets the  following facts scattered in the literature, often in more generality, and then apply them to show the connection with Zilber's conjecture. 
%

\begin{fact}\label{thehammer} Let $T$ be a strongly minimal theory.
\begin{enumerate}[(1)]
\item\label{flatnogroup} \hspace{-0.1cm} {\rm \cite[Lemma 14+remark just after]{Hrustrongmin}}
If the $\mathrm{acl}$-geometry of $T$ is flat, then
$T$ does not interpret an infinite group and $T$ is CM-trivial.
\item \hspace{-0.1cm} {\rm \cite[Theorem 5.1.1]{pillaybook2}} If the $\mathrm{acl}$-geometry of $T$ is locally modular and non-trivial, then
$T$ interprets an infinite group.
%
%
\item \hspace{-0.1cm} {\rm \cite[Lemma 1.6]{Pillayacf}} If $\mathrm{acl}(\emptyset)$ is infinite in $T$, then $T$ admits weak elimination of imaginaries.
\end{enumerate}
\end{fact}

We modify \cite[Lemma 15]{Hrustrongmin} to show our examples have  the characteristic properties of the {\em ab initio} Hrushovski construction.

\begin{conclusion}\label{getflat} For any $\mu \in \mathcal{U}$,
the $\acl$-pregeometry associated with $T_\mu$ is flat (Definition~\ref{defflat}. Thus,
\begin{enumerate}[(1)]
\item\label{flatnogroup} \hspace{-0.1cm}
$T$ does not interpret an infinite group and $T$ is CM-trivial.
%

%
%
\item \hspace{-0.1cm}  If, further, $\mu \in \Fscr$, $T$ has weak elimination of imaginaries.
\end{enumerate}
\end{conclusion}

\begin{proof} Fix $M \models T_\mu$. By Lemma~\ref{dclinacl}, $\acl$ is the same as $\cl_d$. We use the notation of Definition~\ref{allflatdef}  and start with  $\acl$-closed subsets $E_i\leq M$ of finite dimension for $i$ in the finite set $I$.
For flatness, for each $\emptyset \neq S\subseteq I$, let $E_S = \bigcap_{i\in S} E_i$; let  $\check E_S$,  be a finite base for $E_S$.    That is, $\check E_S \leq E_S \leq M$ and $\mathrm{cl}_d(\check E_S) = E_S$.  For $i\in I$, let $F_i = \icl(\bigcup_{i \in S\subseteq I } \check E_S)$. Then, as usual, for $S \subseteq I$ let $F_S = \bigcap_{i\in S} F_i$
and $F_{\emptyset} = \bigcup_{i\in I} F_i$.
Now we have the following

$$E_S = \bigcap_{i\in S} E_i = \acl(\check E_S) = \acl(F_S).$$
The first two equalities are immediate from the definitions. $\check E_S$ is clearly a subset of $F_S$ since $i\in S$ implies $\check E_S \subseteq F_i$. Finally examination of the definitions of $F_i$ and $F_S$ shows $E_S \subseteq \acl(F_S)$ for each $\emptyset \neq S \subseteq I$.
Since $\delta(F_S) =d(F_S)$ and $\delta$ is flat by Lemma~\ref{delta_lemma}.1 applied to the $F_i$ and $F_S$, lifting by  $d(F_S) = d(E_S)$, we have that $d$ is flat.

The consequences follow from Fact~\ref{thehammer}.
\end{proof}

\bibliography{ssgroups}{}
\bibliographystyle{alpha}

\end{document}